\documentclass[11pt]{amsart}
\usepackage[pagewise]{lineno}

\makeatletter
\g@addto@macro{\endabstract}{\@setabstract}
\newcommand{\authorfootnotes}{\renewcommand\thefootnote{\@fnsymbol\c@footnote}}%
\makeatother

\usepackage[margin=1.3in]{geometry}
\usepackage{setspace}
\usepackage{lipsum}
\usepackage{amsfonts}
\usepackage{graphicx}
\usepackage{epstopdf}
\usepackage{algorithmic}
\usepackage{mathrsfs}
\usepackage{enumerate}
\usepackage{todonotes}
\usepackage{amsmath}
\usepackage{caption}
\usepackage{subcaption}
\usepackage{amsopn}
\usepackage{float}
\usepackage{hyperref}
\usepackage[nocompress]{cite}
\usepackage{stmaryrd}
\usepackage[ruled,vlined,noline]{algorithm2e}

\ifpdf
  \DeclareGraphicsExtensions{.eps,.pdf,.png,.jpg}
\else
  \DeclareGraphicsExtensions{.eps}
\fi

\newcommand{\q}{\textbf{q}}

\usepackage{amssymb}

\usepackage{epsfig}  		
\usepackage{epic,eepic}       

 \usepackage[foot]{amsaddr}

\newtheorem{theorem}{Theorem}[section]
\newtheorem{proposition}[theorem]{Proposition}

\theoremstyle{definition}
\newtheorem{definition}[theorem]{Definition}

\theoremstyle{remark}

\newtheorem{remark}[theorem]{Remark}
\newtheorem{assumption}{Assumption}

\numberwithin{equation}{section}

\allowdisplaybreaks
\title[Contraction Estimates for Stochastic Gradient Langevin Integrators]{Contraction Rate Estimates of Stochastic Gradient Kinetic Langevin Integrators}

\begin{document}
\maketitle
\begin{center}
  \normalsize
  \authorfootnotes
   Benedict Leimkuhler\footnote{b.leimkuhler@ed.ac.uk}\textsuperscript{1},
   Daniel Paulin \footnote{dpaulin@ed.ac.uk}\textsuperscript{1},
   Peter A. Whalley\footnote{p.a.whalley@sms.ed.ac.uk}\textsuperscript{1} \par \bigskip

  \textsuperscript{1}School of Mathematics, University of Edinburgh, Edinburgh EH9 2NX, Scotland \par
 

\end{center}



%
%
%

\begin{abstract}
In previous work, we introduced a method for determining convergence rates for integration methods for the kinetic Langevin equation for $M$-$\nabla$Lipschitz $m$-log-concave densities [arXiv:2302.10684, 2023].
In this article, we exploit this method to treat several additional schemes including the method of Brunger, Brooks and Karplus (BBK) and stochastic position/velocity Verlet.  We introduce a randomized midpoint scheme for kinetic Langevin dynamics, inspired by the recent scheme of Bou-Rabee and Marsden  [arXiv:2211.11003, 2022]. We also extend our approach to stochastic gradient variants of these schemes under minimal extra assumptions.  We provide convergence rates of $\mathcal{O}(m/M)$, with explicit stepsize restriction, which are of the same order as the stability thresholds for Gaussian targets and are valid for a large interval of the friction parameter. We compare the contraction rate estimates of many kinetic Langevin integrators from molecular dynamics and machine learning.  Finally we present numerical experiments for a Bayesian logistic regression example.
\end{abstract}



\section{Introduction}

Efficient sampling of high dimensional probability distributions is required for applications such as Bayesian inference and molecular dynamics (see for example \cite{gelman2013bayesian}  and \cite{bond2007molecular,leimkuhler2015molecular}). A popular approach is to employ a Markov chain constructed by discretizing a stochastic differential equation (SDE) and to approximate observable averages using the central limit theorem.  Some common choices of SDEs include overdamped Langevin dynamics \cite{rossky1978brownian,besag1994discussion,roberts1996exponential} and  underdamped/kinetic Langevin dynamics (see below). Other popular methods for MCMC are based on combining Hamiltonian dynamics with Metropolis-Hastings accept/reject condition or combining simple dynamics with stochastic refreshments, which can be simulated exactly \cite{peters2012rejection,bouchard2018bouncy,bierkens2019zig}. 

The focus of this article is on kinetic Langevin dynamics which is the stochastic differential equation system defined by
\begin{equation}\label{eq:underdamped_langevin}
    \begin{split}
    dX_{t} &= V_{t}dt,\\
    dV_{t} &= -\nabla U(X_{t}) dt - \gamma V_{t}dt + \sqrt{2\gamma}dW_{t},
\end{split}
\end{equation}
where $X_{t}, V_{t} \in \mathbb{R}^{n}$, $U$ is a ``potential energy'' function and may be taken to be $-\log \pi$ for a suitable target density $\pi$, $\gamma > 0$ is a friction parameter and $W_{t}$ is the driving $n$-dimensional Brownian motion. It can be shown under mild assumptions on the potential $U$ that the invariant measure of this process is proportional to $\exp{\left(-U(X) - \frac{1}{2}\|V\|^{2} \right)}$ \cite{pavliotis2014stochastic}.  Particle masses and a temperature parameter are typically included in the context of molecular dynamics, which we neglect here in order to simplify the presentation of results; if desired our analysis could easily be modified to include them.  Overdamped Langevin dynamics is the high friction limit of kinetic Langevin dynamics following a time rescaling \cite{pavliotis2014stochastic}.

In this article we focus our attention to proving convergence rates of numerical methods for kinetic Langevin sampling in Wasserstein distance (see \cite{vaserstein1969markovian}). We use coupling methods to establish these convergence rates (see \cite{griffeath1975maximal}), more specifically synchronous coupling as in \cite{leimkuhler2023contraction,sanz2021wasserstein,dalalyan2020sampling,cheng2018underdamped,monmarche2022hmc,monmarche2021high}. Demonstrating that a certain coupling leads to contraction has become a widely used method for demonstrating convergence in terms of Wasserstein distance in both continuous-time scenarios and when discretizing Langevin dynamics and Hamiltonian Monte Carlo \cite{bou2020coupling,bou2022unadjusted,bou2023mixing,eberle2019couplings,deligiannidis2021randomized,riou2022metropolis,schuh2022global,leimkuhler2023contraction}. The numerical methods we consider are the Euler-Maruyama discretisation (EM), the Brunger-Brooks Karplus discretisation (BBK) \cite{brunger1984stochastic}, the stochastic position and velocity Verlet (SPV, SVV) \cite{melchionna2007design}, popular splitting methods including BAOAB and OBABO \cite{PhysRevE.75.056707,leimkuhler2013rational}, a randomized method based on the Hamiltonian integrator of \cite{bou2022unadjusted} (rOABAO) and the stochastic Euler scheme (SES/EB) \cite{chandrasekhar1943stochastic,ermak1980numerical,skeel2002impulse}.

When considering MCMC methods the performance of a sampling scheme is often assessed by measuring the number of steps needed to achieve a certain level of accuracy in the Wasserstein distance metric.  By combining the results of this paper  with  estimates of the stepsize-dependent bias of the numerical methods, it is possible to develop such non-asymptotic bounds in Wasserstein distance which can ultimately provide insight into the computational complexity, convergence rate, and accuracy of the sampling scheme.    Bias estimates of some relevant numerical methods have been treated in \cite{monmarche2021high,monmarche2022hmc,sanz2021wasserstein}.

The aim of this article is to extend the results of \cite{leimkuhler2023contraction} to obtain Wasserstein convergence estimates for a wide interval of the friction parameter whilst maintaining reasonable assumptions on the stepsize. We also propose here a new sampling scheme based on the randomized midpoint method of \cite{bou2022couplings} for Hamiltonian Monte Carlo and we provide convergence results. Moreover, we discuss the use of stochastic gradients and how these proofs can be extended to that setting, which is particularly important in the context of machine learning. We demonstrate our results on an anisotropic Gaussian as well as a Bayesian logistic regression problem involving the MNIST dataset. Although we only treat convergence towards the invariant measure of the scheme in this article, we demonstrate the bias of the methods and discuss these results in combination with the convergence results achieved. We verify the convergence results for the anisotropic Gaussian example, by computing spectral gaps for the numerical methods.

\begin{table}[H]
\begin{center}
\begin{tabular}{ |c|c|c| } 
\hline
Algorithm & stepsize restriction & one-step contraction rate \\
\hline
EM  & $\mathcal{O}(1/\gamma)$ & $\mathcal{O}(mh/\gamma)$\\
BBK & $\mathcal{O}(1/\gamma)$ & $\mathcal{O}(mh/\gamma)$\\
SPV & $\mathcal{O}(1/\gamma)$ & $\mathcal{O}(mh/\gamma)$\\
SVV & $\mathcal{O}(1/\gamma)$ & $\mathcal{O}(mh/\gamma)$\\
BAOAB &$\mathcal{O}((1 - \eta)/\sqrt{M})$ \cite{leimkuhler2023contraction} & $\mathcal{O}(mh^{2}/(1-\eta))$ \cite{leimkuhler2023contraction,monmarche2022hmc} \\ 
OBABO &$\mathcal{O}((1 - \eta)/\sqrt{M})$ \cite{leimkuhler2023contraction}& $\mathcal{O}(mh^{2}/(1-\eta))$ \cite{leimkuhler2023contraction} \\ 
rOABAO &$\mathcal{O}((1 - \eta)/\sqrt{M})$ & $\mathcal{O}(mh^{2}/(1-\eta))$ \\ 
SES/EB &$\mathcal{O}(1/\gamma)$ \cite{sanz2020contractivity,leimkuhler2023contraction} & $\mathcal{O}(mh/\gamma)$ \cite{sanz2020contractivity,leimkuhler2023contraction,dalalyan2020sampling} \\
\hline
\end{tabular}

\end{center}
\caption{The table provides our stepsize restrictions and optimal contraction rates of the discretized kinetic Langevin dynamics with stepsize $h$ for an $m$-convex, $M$-$\nabla$ Lipschitz potential and previous results of \cite{leimkuhler2023contraction} and other recent work for further integrators for comparison. We define $\eta = e^{-\gamma h}$.}
\label{Table:results}
\end{table}

Using full gradients at each iteration can be computationally expensive in case of  large datasets. The predominant approach used in machine learning optimization is to rely on stochastic approximation of the gradient (see e.g. \cite{robbins1951stochastic} for one of the first applications of such approaches). In the context of sampling, there has been a great deal of interest to also use such ideas to improve the scalability of MCMC to large datasets, see e.g. \cite{welling2011bayesian}, or the recent review paper \cite{nemeth2021stochastic}.
Our contribution here is to generalize the contraction rates for all schemes in Table \ref{Table:results} to appropriate versions of these schemes using stochastic dynamics (see Table \ref{Table:stoc_constants} for a summary of results). 
We allow for a flexible choice of unbiased gradient estimators (i.e they do not necessarily have to be based on subsampling) and control errors via expected variability in the Jacobian of the stochastic gradient versus the Hessian of the true potential. It turns out that for all schemes, there is some reduction in convergence rate as the gradient noise increases (we observed this in our numerical experiments when using sub-sampling with very small batch sizes). Nevertheless, for a fixed level of gradient noise, the relative reduction in the contraction rate due to stochastic gradients becomes negligible as the stepsize decreases.

\section{Assumptions and definitions}
\subsection{Assumptions on the potential} \label{sec:assumptions}
We place assumptions on the target measure and the resulting potential $U$. These are strong but allow us to easily obtain quantitative convergence rates.    We assume that the potential has a $M$-Lipschitz gradient and is $m$-strongly convex, which is equivalent to the following assumptions on the Hessian of $U$:

\begin{assumption}[$M$-$\nabla$Lipschitz and $m$-convex] For all $x \in \mathbb{R}^{n}$ there exists $m$ and $M$ such that  $0 < m < M < \infty$, and
\[
m I_{n} \prec \nabla^{2}U(x) \prec M I_{n}.
\]
\end{assumption}

These assumptions are widely used in analysis of optimisation and sampling methods for gradient descent (see \cite{boyd2004convex,dalalyan2017theoretical,durmus2017nonasymptotic}). Analysis of sampling algorithms in the non-convex setting has also been studied (see \cite{cao2019explicit,eberle2019couplings,majka2020nonasymptotic}).

\subsection{Modified Euclidean Norms}\label{Sec:Quadratic_Norm}
We introduce a modified Euclidean norm as in \cite{monmarche2021high} to establish convergence of the discretizations of kinetic Langevin dynamics. It is not possible to establish convergence using the standard Euclidean norm.  We introduce the modified Euclidean norm for $z = (x,v) \in \mathbb{R}^{2n}$ 
\[
\left|\left| z \right|\right|^{2}_{a,b} = \left|\left| x \right|\right|^{2} + 2b \left\langle x,v \right\rangle + a \left|\left| v \right|\right|^{2},
\]
for $a,b > 0$ which is equivalent to the Euclidean norm as long as $b^{2}<a$ with explicit constants given by
\[
\frac{1}{2}\|z\|^{2}_{a,0} \leq \|z\|^{2}_{a,b} \leq \frac{3}{2}\|z\|^{2}_{a,0}.\]

\subsection{Wasserstein Distance} \label{sec:wasserstein_def}

We introduce a notion of distance between probability measures to measure convergence. The metric we
consider is the $p$-Wasserstein distance on $\mathcal{P}_{p}\left(\mathbb{R}^{2n}\right)$ the space of probability measures with finite $p$-th moment, for $p \in \left[0,\infty\right)$. For probability measures $\mu, \nu \in \mathcal{P}_{p}\left(\mathbb{R}^{2n}\right)$ the $p$-Wasserstein distance with respect to the modified norm $\|\cdot\|_{a,b}$ (introduced in Section \ref{Sec:Quadratic_Norm}) is defined by
\[\mathcal{W}_{p,a,b}\left(\nu,\mu\right) = \left( \inf_{\xi \in \Gamma\left( \nu, \mu \right)}\int_{\mathbb{R}^{2d}}\|z_{1} - z_{2}\|^{p}_{a,b}d\xi\left(z_{1},z_{2}\right)\right)^{1/p},\]
where $\Gamma\left(\mu,\nu\right)$ is the set of all couplings between $\mu$ and $\nu$ (the set of joint probability measures with marginals $\mu$ and $\nu$).

If one is able to show contraction of any coupling of two paths of a numerical integrator for kinetic Langevin dynamics then one also obtains convergence in $p$-Wasserstein distance, as it is the infimum over all such couplings (see \cite{cheng2018convergence}[Corollary 7] or \cite{monmarche2021high}[Corollary 20]). The coupling technique we consider for quantitative contraction rates is synchronous coupling, where the two paths are generated using identical noise increments. This coupling is also used in \cite{cheng2018convergence,monmarche2020almost,dalalyan2020sampling}. 

\begin{proposition}\label{prop:Wasserstein}\cite{leimkuhler2023contraction}[Propostion 2.3]
Assume a numerical scheme for kinetic Langevin dynamics with a $m$-strongly convex $M$-$\nabla$Lipschitz potential $U$ and transition kernel $P_{h}$. If any two synchronously coupled chains with initial conditions $\left(x_{0},v_{0}\right)\in \mathbb{R}^{2n}$ and $\left(\Tilde{x}_{0},\Tilde{v}_{0}\right) \in \mathbb{R}^{2n}$ of the numerical scheme have the contraction property 
\begin{equation}\label{eq:contraction_inequality}
    \|(x_{k} - \Tilde{x}_{k},v_{k} - \Tilde{v}_{k})\|^{2}_{a,b} \leq C(1 - c\left(h\right))^{k}\|(x_{0} - \Tilde{x}_{0},v_{0} - \Tilde{v}_{0})\|^{2}_{a,b},
\end{equation}
for $\gamma^{2} \geq C_{\gamma}M$, $h \leq C_{h}\left(\gamma,\sqrt{M}\right)$ and $a,b >0$ such that $b^{2} > a$. Then we have that for all $\gamma^{2} \geq C_{\gamma}M$, $h \leq C_{h}\left(\gamma,\sqrt{M}\right)$, $1 \leq p \leq \infty$, $\mu,\nu \in \mathcal{P}_{p}(\mathbb{R}^{2n})$, and  all $k \in \mathbb{N}$,
\[
\mathcal{W}^{2}_{p}\left(\nu P^k_{h} ,\mu P^k_{h} \right) \leq 3C\max{\left\{a,\frac{1}{a}\right\}}\left(1 - c\left(h\right)\right)^{k} \mathcal{W}^{2}_{p}\left(\nu,\mu\right).
\]
Further to this, $P_{h}$ has a unique invariant measure which depends on the stepsize, $\pi_{h}$, where $\pi_{h} \in \mathcal{P}_{p}(\mathbb{R}^{2n})$ for all $1 \leq p \leq \infty$.
\end{proposition}
\begin{proof}
 The proof is given in \cite{monmarche2021high}[Corollary 20], which relies on \cite{villani2009optimal}[Corollary 5.22, Theorem 6.18].

\end{proof}

\section{Proof Strategy}
We use the proof strategy introduced in \cite{leimkuhler2023contraction} to prove contraction of the numerical schemes considered in this article. We summarize the proof strategy and for further details we refer the reader to \cite{leimkuhler2023contraction}. Our method relies on proving contraction of a ``twisted Euclidean norm" or modified norm (as stated in Section \ref{Sec:Quadratic_Norm}) which is equivalent to the standard Euclidean norm with an explicit constant. The approach is to find constants $a, b > 0$ such that $b^{2} < a$ and the contraction property
\begin{equation}\label{eq:cont_1}
  \|\Tilde{z}_{k+1} - z_{k+1}\|^{2}_{a,b} < \left(1 - c\left(h\right)\right)\|\Tilde{z}_{k} - z_{k}\|^{2}_{a,b},  
\end{equation}
holds with explicit weak assumptions on the parameters $\gamma$ and $h$. Now this is equivalent to showing that
\begin{equation}\label{eq:contraction_matrix_form}
 \overline{z}^{T}_{k}\left(\left(1 - c\left(h\right)\right)M- P^{T}MP\right )\overline{z}_{k} > 0,    \quad \textnormal{where} \quad M = \begin{pmatrix}
    1 & b \\
    b & a
\end{pmatrix},
\end{equation}
and $\overline{z}_{k+1} = P\overline{z}_{k}$. $P$ is determined by the scheme and implicitly depends on $z_{k}$ and $\Tilde{z}_{k}$ through a mean value theorem and the Hessian of the potential. If this relation holds for any choice of $z_{k}$ and $\Tilde{z}_{k}$ then it implies contraction. Therefore proving contraction with a rate $c(h)$ is equivalent to proving that the matrix $\mathcal{H} :=  \left(1 - c(h)\right)M - P^{T}MP \succ 0$ is positive definite. We can use the symmetric structure of the  matrix $\mathcal{H}$ 
\begin{equation}\label{eq:contraction_matrix}
\mathcal{H} = \begin{pmatrix}
    A & B \\
    B & C
\end{pmatrix},    
\end{equation}
to show that $\mathcal{H}$ is positive definite by applying the following Prop. \ref{Prop:PD}.
\begin{proposition}[\cite{leimkuhler2023contraction}]\label{Prop:PD}
Let $\mathcal{H}$ be a symmetric matrix of the form (\ref{eq:contraction_matrix}), then $\mathcal{H}$ is positive definite if and only if $A \succ 0$ and $C - BA^{-1}B \succ 0$. Further if $A$, $B$ and $C$ commute then $\mathcal{H}$ is positive definite if and only if $A\succ 0$ and $AC - B^{2} \succ 0$.
\end{proposition}
\begin{proof}
Proof given in \cite{leimkuhler2023contraction}.
\end{proof}
\section{Numerical Integrators}\label{sec:numerical_schemes}
We will consider several popular numerical integrators for kinetic Langevin dynamics, arising in the literatures of molecular dynamics and  machine learning. The numerical methods are generally defined by $(x_k,v_k) \in \mathbb{R}^{n} \times \mathbb{R}^{n}$ for $k \in \mathbb{N}$ with initial conditions $(x_0,v_0) \in \mathbb{R}^{n} \times \mathbb{R}^{n}$ and given noise sequences.  

Choices for the algorithm include:
\begin{itemize}
    \item the Euler-Maruyama discretization (EM)
    \item splitting methods based on breaking the dynamics into parts which can be solved analytically (in the weak sense) \cite{PhysRevE.75.056707,leimkuhler2013rational,leimkuhler2016computation}
    \item the stochastic Euler scheme (SES/EB) (see \cite{chandrasekhar1943stochastic,ermak1980numerical,skeel2002impulse}), which is popular in the machine learning literature (see \cite{cheng2018underdamped, dalalyan2020sampling,sanz2021wasserstein}) and is based on keeping the force constant and integrating exactly over the interval
    \item the Brunger-Brooks Karplus (BBK) scheme which uses a leapfrog-like approach to propagate position and velocity components, combined with implicit and explicit Euler steps in velocity \cite{brunger1984stochastic,finkelstein2020comparison}
    \item the stochastic position and velocity Verlet schemes (SPV,SVV) based on integrating the force and the OU process together in a splitting scheme introduced in \cite{melchionna2007design}
    \item a new randomized midpoint method based on a Hamiltonian integrator from \cite{bou2022unadjusted}
\end{itemize}
We recommend \cite{durmus2021uniform} and \cite{finkelstein2020comparison} for an introduction to many of these schemes. We next describe these algorithms by giving their respective update rules.
\subsection{Euler-Maruyama}
The Euler-Maruyama discretization with initial conditions $(x_{0},v_{0}) \in \mathbb{R}^{2n}$ and iterations $(x_{k},v_{k}) \in \mathbb{R}^{2n}$ for $k \in \mathbb{N}$ are defined by the update rule
\begin{equation}
\begin{split}
x_{k+1} & =  x_k + hv_k,\\
v_{k+1} & = v_k - h\nabla U(x_k) - h\gamma v_k + \sqrt{2\gamma h}\xi_{k+1},
\end{split}
\end{equation}
where $(\xi_k)_{k\in\mathbb{N}}$ are independent ${\mathcal N}(0,1)$ draws. 
\subsection{Splitting Methods}
Integrators studied in \cite{PhysRevE.75.056707} and \cite{leimkuhler2013rational} are defined by splitting the dynamics into parts given by $\mathcal{B}$ (integrating the velocity by the force), $\mathcal{A}$ (integrating the position by the velocity) and $\mathcal{O}$ (the solution in the weak sense to the OU process) with update rules given by
\begin{equation}\label{eq:BAO}
\begin{split}
    &\mathcal{B}: v \to v - h\nabla U(x), \\
    &\mathcal{A}: x  \to x + hv,\\
    &\mathcal{O}: v \to \eta v + \sqrt{1 - \eta^{2}}\xi,
\end{split}  
\end{equation}
where 
\[
\eta := \exp{\left(-\gamma h \right)}.
\] 
Then the schemes we will study in this framework are BAOAB and OBABO (the ordering given from left to right based on their application as operators defined in \eqref{eq:BAO}). When there is a repeated letter in the ordering that means that each operator is taken with half a step, i.e. $h \to h/2$, $\eta \to \eta^{1/2}$. For example BAOAB performs two half steps of $\mathcal{B}$ and $\mathcal{A}$ and one full step of $\mathcal{O}$. For computational efficiency, in practice the same gradient evaluation for the last $\mathcal{B}$ step in a BAOAB iteration is used for the first $B$ step in the next iteration, as the position is not updated.

\subsection{Stochastic Euler Scheme}
The stochastic Euler scheme is based on keeping the force constant and integrating the dynamics exactly over the time interval. For initial conditions $(x_{0},v_{0}) \in \mathbb{R}^{2n}$, the iterations $(x_{k},v_{k}) \in \mathbb{R}^{2n}$ for $k \in \mathbb{N}$ are defined by the update rule
\begin{equation}\label{eq:SES}
    \begin{split}
    x_{k+1} &= x_{k} + \frac{1-\eta}{\gamma}v_{k} - \frac{\gamma h + \eta -1}{\gamma^{2}}\nabla U\left(x_{k}\right) + \zeta_{k+1},\\
    v_{k+1} &= \eta v_{k} - \frac{1 - \eta}{\gamma} \nabla U\left(x_{k}\right) + \omega_{k+1},
    \end{split}
\end{equation}
where $\eta := \exp{\left(-\gamma h \right)}$ and 
\[
\zeta_{k+1} = \sqrt{2\gamma}\int^{h}_{0}e^{-\gamma\left( h - s\right)}dW_{h\gamma + s}, \qquad \omega_{k+1} = \sqrt{2\gamma} \int^{h}_{0} \frac{1 - e^{-\gamma\left( h - s\right)}}{\gamma}dW_{h\gamma + s}.
\]
$\left(\zeta_{k},\omega_{k}\right)_{k \in \mathbb{N}}$ are i.i.d Gaussian random vectors with covariances matrix $\Sigma \otimes I_{n}$ with $\Sigma$ given by
\begin{equation}\label{eq:SES_covariance}
\Sigma = \begin{pmatrix}
    \Sigma_{1} & \Sigma_{2} \\
    \Sigma_{2} & \Sigma_{3}
\end{pmatrix},
\end{equation}
where
\begin{align*}
    \Sigma_{1} &= \frac{1}{\gamma}\left(2h - \frac{3 - 4\eta + \eta^{2}}{\gamma}\right),\\
    \Sigma_{2} &= \frac{1}{\gamma}\left(1-\eta\right)^{2},\\
    \Sigma_{3} &= 1 - \eta^{2},\\
\end{align*}
as defined in \cite{durmus2021uniform}. 

All of the schemes mentioned so far were studied in \cite{leimkuhler2023contraction}.  We now introduce some additional schemes, which include the BBK integrator \cite{brunger1984stochastic} which is popular in molecular dynamics and the stochastic position and velocity Verlet methods \cite{melchionna2007design}.  

\subsection{BBK}
For initial conditions $(x_{0},v_{0}) \in \mathbb{R}^{2n}$, the iterations $(x_{k},v_{k}) \in \mathbb{R}^{2n}$ for $k \in \mathbb{N}$ of the BBK method of \cite{brunger1984stochastic} are defined by the update rule
\begin{align*}
    x_{k+1} &= x_{k} + h(1-\frac{\gamma h}{2})v_{k} - \frac{h^{2}}{2}\nabla U(x_{k}) + \sqrt{2\gamma}\frac{h^{3/2}}{2}\xi_{k},\\
    v_{k+1} &= \frac{1 - \gamma h /2}{1 + \gamma h / 2}v_{k} - \frac{h}{2(1 + \gamma h/2)}\left(\nabla U(x_{k}) + \nabla U(x_{n+1})\right) + \frac{\sqrt{2\gamma h}}{2(1 + \gamma h/2)}(\xi_{k+1} + \xi_{k}),
\end{align*}
which can be rewritten as \cite{finkelstein2020comparison}
\begin{align*}
    &(B_{1}) \qquad v_{k+1/2} = v_{k} + \frac{h}{2}\left(-\nabla U(x_{k}) - \gamma v_{k} + \frac{\sqrt{2\gamma}}{\sqrt{h}}\xi_{k}\right),\\
    &(A) \qquad x_{k+1} = x_{k} + hv_{k+1/2},\\
    &(B_{2}) \qquad v_{k+1} = v_{k+1/2} + \frac{h}{2}\left(-\nabla U(x_{k+1}) - \gamma v_{k+1} + \frac{\sqrt{2\gamma}}{\sqrt{h}}\xi_{k+1}\right).
\end{align*}
This can be viewed as an explicit Euler step followed by a position update followed by an implicit Euler step. We denote the explicit Euler step by $B_{1}$ and the implicit Euler step by $B_{2}$

\subsection{Stochastic Position and Velocity Verlet}
The stochastic position and velocity Verlet schemes are defined through an alternative splitting of the dynamics based on keeping the $B$ and $O$ steps together in an exact integration. We define the operators involved in the update rule by
\begin{align*}
    \mathcal{V}_{s}(h)&: v \to \eta v - \frac{1 - \eta}{\gamma} \nabla U(x) + \sqrt{1 - \eta^{2}} \xi,\\
    \mathcal{A}(h)&: x \to x + h v,
\end{align*}
where $\eta = e^{-\gamma h}$. Then the stochastic position Verlet is defined by $\mathcal{A}(h/2)\mathcal{V}_{s}(h)\mathcal{A}(h/2)$ and the stochastic velocity Verlet is defined by $\mathcal{V}_{s}(h/2)\mathcal{A}(h)\mathcal{V}_{s}(h/2)$.

\subsection{Randomized midpoint method}
Other algorithms for kinetic Langevin dynamics include the randomized midpoint methods considered in \cite{shen2019randomized} and analyzed in \cite{cao2020complexity}, which have improved dimension dependence in non-asymptotic estimates. However they involve multiple gradient evaluations at each step and are not able to be analyzed in our framework; this problem has been discussed in \cite{sanz2021wasserstein}. For contractivity of algorithms involving several gradient evaluations we refer the reader to \cite{sanz2020contractivity}.

In the recent paper \cite{bou2022unadjusted} the authors consider such a method for Hamiltonian Monte Carlo, whose discretization is closely related to the OBABO or OABAO discretization in the $\gamma \to \infty$ limit \cite{monmarche2022hmc}. More precisely one could consider the following procedure.   

Fix a stepsize $h$ then sample $u \sim [0,h]$ and compute
\begin{align*}
    &\mathcal{A}: x \to x + uv, \\
    &\mathcal{B}: v \to v - h\nabla U(x),\\
    &\mathcal{A}: x \to x + (h-u)v,
\end{align*}
which is the following update
\begin{align*}
    x_{k+1} &= x_{k} + h v_{k} -h(h-u)\nabla U(x_{k} + uv_{k}),\\
    v_{k+1} &= v_{k} - h\nabla U(x_{k} + uv_{k}),
\end{align*}
then only considering the randomness in the gradient evaluation we arrive at the Verlet scheme considered in \cite{bou2022unadjusted}. We define $r\mathcal{ABA}$ to be the update
\begin{align*}
    x_{k+1} &= x_{k} + hv_{k}  - \frac{h^{2}}{2}\nabla U (x_{k} + uv_{k}),\\
    v_{k+1} &= v_{k} - h\nabla U (x_{k} + u v_{k}),
\end{align*}
where $u \sim \mathcal{U}(0,h)$ as introduced in \cite{bou2022unadjusted}. The key difference being that the gradient is evaluated at a random midpoint in the interval of numerical integration. We define the kinetic Langevin dynamics integrator rOABAO to be $\mathcal{O}\left(r\mathcal{ABA} \right)\mathcal{O}$.

We remark that we can achieve contraction rates by  coupling two trajectories which have common noise (in Brownian increment and randomized midpoint) $\left(\zeta_{k},u_{k}\right)_{k \in \mathbb{N}}$
with the previously introduced methods. The convergence rates will be established in Section \ref{Sec:Proofs}.

\section{Convergence Rates}

\begin{theorem}\label{Theorem:Total_theorem}
For the numerical schemes for kinetic Langevin dynamics given in Table \ref{Table:constants} with an $m$-strongly convex, $M$-$\nabla$Lipschitz potential $U$ we consider any sequence of synchronously coupled random variables with initial conditions $(x_{0},v_{0}) \in \mathbb{R}^{2n}$ and $(\Tilde{x}_{0},\Tilde{v}_{0}) \in \mathbb{R}^{2n}$.

\begin{table}[h]

\begin{center}
\begin{tabular}{ |c|c|c|c|c|c|c| } 
\hline
{\rm Algorithm} & $h_{0}$ & $\gamma_{0}$  & b & c(h) & C & s\\
\hline
{\rm EM}  & $1/2\gamma $ & $2\sqrt{M}$  & $1/\gamma$ & $mh/2\gamma$ & $1$ & $k$\\
{\rm BBK} & $1/4\gamma$ & $\sqrt{12M}$ & $h/2 + 1/\gamma$ & $mh/4\gamma$ & $7$ & $k-1$\\
{\rm SPV} & $1/2\gamma$ & $\sqrt{11M}$ & $h/(1-\eta)$ & $mh/4\gamma$ & $7$ & $k-1$\\
{\rm SVV} & $1/2\gamma$ & $\sqrt{11M}$ & $h/(1-\eta)$ & $mh/4\gamma$ & $7$ & $k-1$\\
{\rm BAOAB} & $(1-\eta)/2\sqrt{M}$& \textnormal{implicit} & $h/(1-\eta)$& $h^{2}m/4(1-\eta)$&$7$ & $k - 1$\\
{\rm OBABO} &$(1-\eta)/4\sqrt{M}$& \textnormal{implicit}  & $h/(1-\eta)$ & $h^{2}m/4(1-\eta)$ &$7$ & $k-1$\\
{\rm rOABAO} & $(1-\eta)/2\sqrt{M}$ & \textnormal{implicit} & $h/(1-\eta)$ & $h^{2}m/4(1-\eta)$ & $7$ & $k-1$\\ 
{\rm SES/EB} & $1/2\gamma$ & $5\sqrt{M}$ & $1/\gamma$ & $mh/4\gamma$ & $1$ & $k$\\
\hline
\end{tabular}
\end{center}

\caption{Constants for contraction of each scheme for contraction. Implicit refers to the implicit assumption on $\gamma$ through the stepsize restriction $h_{0}$, the value of $\gamma^2$ must be greater than a certain constant multiple of $M$. For example for $h_{0} = (1-\eta)/\alpha\sqrt{M}$ is satisfied when $\gamma \geq 2\alpha\sqrt{M}$ and $h <1/(2\gamma)$ and for $h \geq 1/(2\gamma)$ we have $h_{0} \geq 1/(6\alpha\sqrt{M})$.}
\label{Table:constants}
\end{table}

Under stepsize restrictions $h < h_{0}$ and $\gamma \geq \gamma_{0}$ for constants given in Table \ref{Table:constants} we have the contraction 
\[
\|(x_{k} - \Tilde{x}_{k},v_{k} - \Tilde{v}_{k})\|_{a,b} \leq C(1-c(h))^{s/2}\|(x_{0} - \Tilde{x}_{0},v_{0} - \Tilde{v}_{0})\|_{a,b},
\]
with norm given with constants $a = 1/M$ and $b$, where $b$, the contraction rate $c(h)$, the preconstant $C$ and the number of steps $s$ are given in Table \ref{Table:constants} and are specific to each scheme.

\end{theorem}
The proofs for EM, BAOAB, OBABO, SES are given in \cite{leimkuhler2023contraction}; for all other schemes the proofs are given here.

Referring to Table \ref{Table:constants} we have that the convergence rate $c(h)$ is proportional to $m/\gamma$ for small $h$, which is shown to match the convergence rate of the continuous dynamics for large $\gamma$ (see for example \cite{dalalyan2020sampling}). We have that the convergence rate is $hm/\gamma$ for all the schemes apart from BAOAB, OBABO and rOABAO, which have convergence rates which are faster than the continuous dynamics for large values of $\gamma$ and $h$ (as originally noted in \cite{leimkuhler2023contraction}). This is due to the fact that the $\mathcal{O}$ step is integrated exactly separately and one can take the high friction limit. Since the $\mathcal{O}$ step also leaves the measure invariant the bias in these types of schemes comes from the Hamiltonian integrator and hence retain high order perfect sampling bias \cite{leimkuhler2016computation}. However, these splitting schemes are only strong order $1$ and this can be seen particularly for large values of friction when the convergence rates are higher than for the continuous dynamics. (It fails to approximate the continuous dynamics, but it is accurate in the sampling context \cite{leimkuhler2016computation}).

\section{Proof of Theorem \ref{Theorem:Total_theorem}}\label{Sec:Proofs}
In the following contraction rate proofs we follow the structure of \cite{leimkuhler2023contraction} for the additional schemes that we are analysing.

\begin{proof}[Proof for rOABAO]
Using the fact that $\left(r\mathcal{OABAO}\right)^{n} = \mathcal{O}\left(r\mathcal{ABAO}\right)^{n-1}r\mathcal{ABAO}$, we can instead proof contraction of $r\mathcal{ABAO}$. This is done to simplify the problem to one gradient evaluation per step. Denoting two synchronous realisations (synchronous in the sense of $\left(u_{k}\right)_{k \in \mathbb{N}}$ and $\left(\xi_{k}\right)_{k \in \mathbb{N}}$) of rABAO as $\left(x_{j},v_{j}\right)$ and $\left(\Tilde{x}_{j}, \Tilde{v}_{j}\right)$ for $j \in \mathbb{N}$.
Further we denote $\overline{x} := \left(\Tilde{x}_{j} - x_{j}\right)$, $\overline{v} = \left(\Tilde{v}_{j} - v_{j}\right)$ and $\overline{z} = \left(\overline{x}, \overline{v}\right)$, where $\overline{z}_{j} = \left(\overline{x}_{j}, \overline{v}_{j}\right)$ for $j = k,k +1$ for $k \in \mathbb{N}$, where $\overline{z}_{k}$ are defined by the update rule
\begin{align*}
    \overline{x}_{k+1} &= \overline{x}_{k} + h \overline{v}_{k} - \frac{1}{2}h^{2}Q\left(\overline{x}_{k} + u_{k}\overline{v}_{k}\right),\quad \overline{v}_{k+1} = \eta\overline{v}_{k} - h\eta Q\left(\overline{x}_{k} + u_{k}\overline{v}_{k}\right).
\end{align*}
$Q$, by mean value theorem, is defined by $Q = \int^{1}_{t = 0}\nabla^{2}U(\Tilde{x}_{k} + u_{k}\Tilde{v}_{k} + t(x_{k} - \Tilde{x}_{k} + u_{k}(v_{k} - \Tilde{v}_{k})))dt$,
then $\nabla U\left(\Tilde{x}_{k} + u_{k}\Tilde{v}_{k}\right) - \nabla U\left(x_{k} + u_{k}v_{k}\right) = Q \left(\overline{x} + u_{k}\overline{v}\right)$. We use the notation of (\ref{eq:contraction_matrix}), where
\begin{align*}
    A &= -c(h) +Q \left(2 b \eta  h+ h^2\right) + Q^2 \left(-a \eta ^2 h^2- b \eta  h^3-\frac{1}{4} h^4\right),\\
    B &= b(1-\eta) - h - bc(h) +Q \left(a \eta ^2 h+\frac{3}{2} b \eta  h^2+b \eta  h u+\frac{1}{2}h^2 u+\frac{1}{2}h^3\right) \\
    &+  Q^2 \left(-a \eta ^2 h^2 u- b \eta  h^3 u- \frac{1}{4}h^4 u\right),\\
    C &= a(1-\eta ^2)-a c(h) -2 b \eta  h-h^2 +Q \left(2 a \eta ^2 h u+3 b \eta  h^2 u+ h^3 u\right)\\
    &+Q^2 \left(-a \eta ^2 h^2 u^2- b \eta  h^3 u^2-\frac{1}{4} h^4 u^2\right),
\end{align*}
for this scheme and $\eta = \exp{\{-\gamma h\}}$, further for the ease of notation we have used $u := u_{k}$. We have chosen $b$, such that $B$ simplifies to 
\begin{align*}
     B &= -bc(h) +  Q \left(a \eta ^2 h+\frac{3}{2} b \eta  h^2+b \eta  h u+\frac{1}{2}h^2 u+\frac{1}{2}h^3\right) +  Q^2 \left(-a \eta ^2 h^2 u- b \eta  h^3 u- \frac{1}{4}h^4 u\right).
\end{align*}
It is sufficient to prove that $A \succ 0$ and that $C - BA^{-1}B \succ 0$, noting that $A$, $B$ and $C$ all commute as they are all polynomials in $Q$. It is therefore sufficient to show that $A \succ 0$ and $AC - B^{2} \succ 0$ and $A$ is positive definite by the proof of \cite{leimkuhler2023contraction}[Theorem 6.1]. We use $P_{AC - B^{2}}\left(\lambda\right)$ to denote the eigenvalues of $AC-B^{2}$, where $\lambda$ are eigenvalues of $Q$, where we know $\lambda \in [m,M]$ due to the assumptions on $U$. $AC - B^{2} \succ 0$ is shown as follows
\begin{align*}
    &\frac{P_{AC - B^{2}}\left(\lambda\right)}{h\lambda} =  \frac{\left(\eta^2 -1\right) c\left(h\right)}{h M \lambda} + h \left( \frac{(1 + \eta)^{2}}{M} -\frac{\eta^{2}\lambda}{M^{2}} + \frac{c\left(h\right)^{2}}{h^{2} M \lambda} 
 - \frac{2c\left(h\right)\eta^{2}u}{hM}\right)\\
    &+ h^{2}\left(\frac{c\left(h\right)}{h} \left(-\frac{1 + 2\eta}{M} + \frac{1}{\lambda} + \frac{2\eta}{(1-\eta)\lambda} + \frac{\eta^{2}\lambda}{M^{2}}\right) + u\left(\frac{\eta^{2}\lambda}{M} + \frac{2\eta^{3}\lambda}{\left(1-\eta\right)M} + \frac{c\left(h\right)\eta^{2}\lambda u}{hM}\right)\right) \\
    &+ h^3 \left(-\frac{\eta^{2}}{\left(1-\eta\right)^{2}} - \frac{3\eta}{\left(1-\eta\right)^{2}} - \frac{1}{1-\eta} - \frac{c(h)^{2}}{h^{2}(1-\eta)^{2}\lambda} - \frac{\lambda (1-\eta^{2})}{4M} - \frac{\eta \lambda}{\left(1 - \eta\right)M}\right) + \\
    &h^{3}u\left(\frac{c(h)}{h}\left(-1 + \frac{2\eta}{\left(1 - \eta\right)^{2}} + \frac{1}{1-\eta} - \frac{3\eta}{1-\eta} - \frac{2\eta^{2}\lambda}{(1-\eta)M}\right) + u\left(-\frac{\lambda}{4} - \frac{\eta^{2}\lambda}{\left(1-\eta\right)^{2}} - \frac{\eta \lambda}{1-\eta}\right)\right)\\
 &+ h^{4}\left(\frac{c(h)}{h}\left(\frac{3\eta}{\left(1-\eta\right)^{2}} + \frac{1}{1-\eta} + \frac{\lambda}{4M} + \frac{\eta \lambda}{\left(1-\eta\right)M}\right)\right)\\
 &+ h^{4}\left(u\left(\frac{3\eta^{2}\lambda}{\left(1-\eta\right)^{2}} + \frac{2\eta\lambda}{1-\eta} + \frac{\lambda}{2\left(1-\eta \right)} + \frac{c\left(h\right)\lambda u}{4h} + \frac{c(h)\eta \lambda u}{h(1-\eta)}\right) \right)\\
 &+ h^{5} \left(-\frac{\eta^{2}\lambda}{4\left(1-\eta\right)^{2}} - \frac{2c(h)\eta\lambda u}{h\left(1 - \eta\right)^{2}} - \frac{c(h)\lambda u}{2h\left(1- \eta\right)} \right) \\
 & \geq -\frac{\left(1 +\eta\right)h}{4M} + h\left(\frac{1 + 3\eta/2}{M}\right) + h^{2}\left( 0 \right)\\
 &+ h^{3}\left( \left(-\frac{3\eta + 5/4}{\left(1 - \eta\right)^{2}} - \frac{1-\eta^{2}}{4}\right)+ u\frac{c(h)}{h}\left(-\frac{2\eta^{2}}{1-\eta}\right) + u^{2}M\left(-\frac{1}{4} - \frac{1}{\left(1-\eta\right)^{2}}  \right)\right)\\
 &+ h^{4}(0) + h^{5}\left(-\frac{\eta^{2}\lambda}{4(1-\eta)^{2}}\right)\\
 &\geq \frac{3h}{4M} + \frac{5\eta h}{4M} + \left(- \frac{3\eta h}{4M} - \frac{5h}{16M} - \frac{h}{16M} - \frac{h}{32M} - \frac{6h}{64M}\right) > 0, 
\end{align*}
where we have imposed the restriction $h < \frac{1- \eta}{\sqrt{4M}}$ and used $u <h$. Therefore contraction of rABAO holds and all computations can be checked using Mathematica. We now bound the remaining terms to achieve a contraction result for rOABAO. We bound $\mathcal{O}$ operator on $\|\cdot \|_{a,b}$ by
\begin{align*}
    &\|\Phi_{\rm{O}}\left(\Tilde{x},\Tilde{v}\right) - \Phi_{\rm{O}}\left(x,v\right)\|^{2}_{a,b}  \leq 3\|\left(\overline{x},\overline{v}\right)\|^{2}_{a,b},
\end{align*}
where we have used the norm equivalence in Section \ref{Sec:Quadratic_Norm}. We bound
\begin{align*}
    &\|\Phi_{\rm{rABAO}}\left(\Tilde{x},\Tilde{v}\right) - \Phi_{\rm{rABAO}}\left(x,v\right)\|^{2}_{a,b}  = \\
    &\|( \overline{x} + h \overline{v} - \frac{1}{2}h^{2}Q\left(\overline{x} + u\overline{v}\right),\eta^{1/2}\overline{v} - h\eta^{1/2} Q\left(\overline{x} + u\overline{v}\right))\|^{2}_{a,b},\\
    &\leq 3 \left(\|\overline{x} + h \overline{v}\|^{2} + \frac{1}{4}h^{4}M^{2}\|\overline{x} + u\overline{v}\|^{2} + a\eta\left(\|\overline{v}\|^{2} + h^{2}M^{2}\|\overline{x} + u\overline{v}\|^{2}\right)\right)\\
    &\leq 14\|(\overline{x},\overline{v})\|^{2}_{a,b},
\end{align*}
and we can combine these estimates to achieve the contraction result of rOABAO.
\end{proof}


\begin{proof}[Proof for BBK]

Using the relation
\[
\left(\Phi_{\mathcal{BBK}}\right)^{n} = \Phi_{B_{2}}\circ \Phi_{A}\circ\left(\Phi_{B_{1}} \circ \Phi_{B_{2}} \circ \Phi_{A}\right)^{n-1} \circ \Phi_{B_{1}},
\]
where
\begin{align*}
    \Phi_{A}(x,v) &= \left(x + hv,v\right),\\
    \Phi_{B_{1}}(x,v) &= \left(x,v + \frac{h}{2}\left(-\nabla U(x) - \gamma v + \frac{\sqrt{2\gamma}}{\sqrt{h}}\xi_{1}\right)\right),\\
    \Phi_{B_{2}}(x,v) &= \left(x,\left(1 + \frac{h}{2}\gamma\right)^{-1}\left(v - \frac{h}{2}\nabla U(x)+ \frac{\sqrt{2\gamma}}{\sqrt{h}}\xi_{2}\right)\right),
\end{align*}
and $\xi_{1},\xi_{2}\sim \mathcal{N}(0,1)^{n}$. We can instead prove contraction of $\Phi_{B_{1}} \circ \Phi_{B_{2}} \circ \Phi_{A}$, by doing this we only have to deal with a single evaluation of the Hessian at each step. We will denote two synchronous realisations of $\Phi_{B_{1}} \circ \Phi_{B_{2}} \circ \Phi_{A}$ as $\left(x_{j},v_{j}\right)$ and $\left(\Tilde{x}_{j}, \Tilde{v}_{j}\right)$ for $j \in \mathbb{N}$.
Further we denote $\overline{x} := \left(\Tilde{x}_{j} - x_{j}\right)$, $\overline{v} = \left(\Tilde{v}_{j} - v_{j}\right)$ and $\overline{z} = \left(\overline{x}, \overline{v}\right)$, where $\overline{z}_{j} = \left(\overline{x}_{j}, \overline{v}_{j}\right)$ for $j = k,k +1$ for $k \in \mathbb{N}$, where $\overline{z}_{k}$ are defined by the update rule 
\begin{align*}
    \overline{x}_{k+1} &= \overline{x}_{k} + h \overline{v}_{k}, \quad \overline{v}_{k+1} = \frac{1 - \frac{h}{2}\gamma}{1 + \frac{h}{2}\gamma}\left(\overline{v}_{k} - \frac{h}{2}Q(\overline{x}_{k} + h\overline{v}_{k})\right) - \frac{h}{2}Q(\overline{x}_{k} + h \overline{v}_{k}),
\end{align*}
where $Q$ is defined by mean value theorem, $Q = \int^{1}_{t = 0}\nabla^{2}U(\Tilde{x}_{k} + h\Tilde{v}_{k} + t(x_{k} - \Tilde{x}_{k} + h(v_{k} - \Tilde{v}_{k})))dt$,
then $\nabla U\left(\Tilde{x}_{k} + h\Tilde{v}_{k}\right) - \nabla U\left(x_{k} + hv_{k}\right) = Q \left(\overline{x} + h\overline{v}\right)$.
We use the notation of (\ref{eq:contraction_matrix}), where
\begin{align*}
    A &= -c(h)+\frac{2b h Q}{\frac{1}{2}\gamma h+1}-\frac{a h^2 Q^2}{(\frac{1}{2}\gamma h+1)^2},\\
    B &= b\frac{\gamma h}{\frac{1}{2} \gamma h+1}-h-bc(h)+ Q \left(\frac{ah\left(1-\frac{1}{2}\gamma h\right)}{(\frac{1}{2} \gamma h+1)^2}+\frac{2 b h^2}{\frac{1}{2} \gamma
   h+1}\right)-\frac{a h^3 Q^2}{(\frac{1}{2} \gamma h+1)^2},\\
    C &= a(1-c(h)) - h^{2} -\frac{a\left(1-
   \frac{1}{2}\gamma h\right)^{2}}{(\frac{1}{2}\gamma h+1)^2} -\frac{2bh(1 - \frac{1}{2} \gamma h)}{\frac{1}{2} \gamma
   h+1} \\
   &+ Q \left(\frac{2ah^2(1 - \frac{1}{2}\gamma h)}{(\frac{1}{2}\gamma h+1)^2}+\frac{2 b h^3}{\frac{1}{2} \gamma
   h+1}\right)-\frac{a h^4 Q^2}{(\frac{1}{2}\gamma h+1)^2}.
\end{align*}
This motivates the choice of $b = (\frac{\gamma h}{2} + 1)/\gamma$. We have chosen $b$ such that $B$ simplifies to 
\begin{align*}
     B &= -bc(h)+ Q \left(\frac{ah\left(1-\frac{1}{2}\gamma h\right)}{(\frac{1}{2} \gamma h+1)^2}+\frac{2 b h^2}{\frac{1}{2} \gamma
   h+1}\right)-\frac{a h^3 Q^2}{(\frac{1}{2} \gamma h+1)^2}.
\end{align*}
It is sufficient to prove that $A \succ 0$ and that $C - BA^{-1}B \succ 0$, noting that $A$, $B$ and $C$ commute as they are all polynomials in $Q$. It is therefore sufficient to prove that $A \succ 0$ and $AC - B^{2} \succ 0$. First considering $A$ with our choice of $c(h)$, $a$ and $b$, we have that
\begin{align*}
    P_{A}(\lambda) &= -c(h) + \frac{2h\lambda}{\gamma} - \frac{h^{2}\lambda^{2}}{M(\frac{1}{2}\gamma h +1)^{2}}\\
    &\geq h\lambda\left(\frac{7}{4\gamma}- h\right) > 0,
\end{align*}
where $h < 
\frac{7}{4\gamma}$. What remains is to show that $AC - B^{2} \succ 0$. This discretization is more complicated that the previous discretisations, we find that expanding $P_{AC-B^{2}}$ in terms of $a$ is convenient to show positive definiteness (as in \cite{leimkuhler2023contraction}). By using for example Mathematica one can check that 
$P_{AC-B^2}(\lambda) = (c_0 + c_1 a + c_2 a^2)/h\lambda$, where
\begin{align*}
    &(c_{1} + c_{2} a)/h\lambda = \frac{2}{\gamma} -\frac{2}{\gamma (\frac{1}{2}\gamma h+1)^2} + \frac{
   h}{(\frac{1}{2}\gamma h+1)^2} +\frac{h(1-\frac{1}{2}
   \gamma h)}{(\frac{1}{2}\gamma h+1)^2} -\frac{2 c(h)}{\gamma}\\
   &+\frac{2 b c(h)(1-\frac{1}{2}\gamma h)}{(\frac{1}{2}\gamma h+1)^2} -\frac{2 h c(h)(1- \frac{1}{2}\gamma h)}{(\frac{1}{2}\gamma h+1)^2}+\frac{2 h^2 \lambda}{\gamma (\frac{1}{2}\gamma h+1)^2} +\frac{c(h)}{h\lambda (\frac{1}{2}\gamma h+1)^2} -\frac{c(h)}{h\lambda}\\
   &-\frac{ \gamma c(h)}{\lambda (\frac{1}{2}\gamma h+1)^2} -\frac{2 b h^2 c(h) \lambda}{(\frac{1}{2}\gamma h+1)^2} +\frac{\frac{1}{4} \gamma^2 h c(h)}{\lambda (\frac{1}{2}\gamma h+1)^2} +\frac{h^3 c(h) \lambda}{(\frac{1}{2} \gamma
   h+1)^2} +\frac{c(h)^2}{h\lambda}\\
   &+ a \left(\frac{h c(h) \lambda}{(\frac{1}{2}\gamma h+1)^2}-\frac{h \lambda}{(\frac{1}{2}\gamma h+1)^2}\right),\\
   &\geq \frac{7
   h}{8(\frac{1}{2}\gamma h+1)^2} -\frac{2 c(h)}{\gamma} -\frac{ \gamma c(h)}{\lambda (\frac{1}{2}\gamma h+1)^2} -\frac{2 b h^2 c(h) \lambda}{(\frac{1}{2}\gamma h+1)^2}\\
   &\geq \frac{64}{81}h\left(\frac{7}{8} - \frac{1}{4} - \frac{h^{2}m\lambda(\frac{\gamma h}{2} + 1)}{2\gamma^{2}}  \right) - \frac{mh}{2\gamma^{2}} \geq \frac{64}{81}h\left(\frac{5}{8} - \frac{1}{64}  \right) - \frac{h}{8} \geq \frac{7h}{20},
\end{align*}
where we have relied on the fact that $\gamma^{2} \geq 4M \geq 4m$, and, for $h < \frac{1}{4\gamma}$, that $1 - \frac{1}{2}\gamma h > \frac{7}{8}$ and $1/(\frac{1}{2}\gamma h + 1) > \frac{8}{9}$.

Now considering the remaining terms we have that 
\begin{align*}
    c_{0} &= -\frac{b^2 c(h)^2}{h\lambda}+\frac{4 b h c(h)}{\gamma}-\frac{4 h}{\gamma^2}-\frac{2 h^2 c(h)}{\gamma}+\frac{2 c(h)}{\gamma \lambda} \geq -\frac{4h}{\gamma^{2}},
\end{align*}
where we can combine this with the previous estimate to obtain
\[
\frac{P_{AC - B^{2}}(\lambda)}{h\lambda} \geq h \left(\frac{7}{20M} - \frac{4}{\gamma^{2}} \right),
\]
which holds for $\gamma \geq \sqrt{12M}$. Hence $AC - B^{2} \succ 0$ and our contraction results hold.

Now we need to compute the prefactors.  First we consider
\begin{align*}
    &\|\Phi_{B_{2}} \circ \Phi_{A}(x,v) - \Phi_{B_{2}} \circ \Phi_{A}(\Tilde{x},\Tilde{v})\|^{2}_{a,b} \leq \left \|\left(\overline{x} + h \overline{v},\left(1 + \frac{h}{2}\gamma\right)^{-1}\left(\overline{v} - \frac{h}{2}Q\left(\overline{x} + h \overline{v}\right)\right)\right) \right \|^{2}_{a,b}\\
    & \hspace{0.5in}\leq 3\left( \|\overline{x}\|^{2}+ h^{2}\|\overline{v}\|^{2} + a\left(1 + \frac{h}{2}\gamma\right)^{-2}\left(\|\overline{v}\|^{2} + \frac{h^{2}}{2}M^{2}\left(\|\overline{x}\|^{2} + h^{2}\|\overline{v}\|^{2}\right)\right)\right)\\
    &\hspace{0.5in}\leq 7 \left \|(\overline{x},\overline{v})\right \|^{2}_{a,b},
\end{align*}
then observe
\begin{align*}
    &\|\Phi_{B_{1}}(x,v) - \Phi_{B_{1}}(\Tilde{x},\Tilde{v})\|^{2}_{a,b} \leq \|\left(\overline{x},\overline{v} - \frac{h}{2}Q\overline{x} -\gamma \frac{h}{2}\overline{v}\right)\|^{2}_{a,b}\\
    &\hspace{0.5in}\leq \frac{3}{2}\left( \|\overline{x}\|^{2} + 2a\left(1-\frac{\gamma h}{2}\right)^{2}\|\overline{v}\|^{2} + a\frac{h^{2}}{2}M^{2}\|\overline{x}\|^{2} \right)\\
    &\hspace{0.5in}\leq 6\|(\overline{x},\overline{v})\|^{2}_{a,b},
\end{align*}
and combine these estimates to achieve the contraction result for BBK.
\end{proof}

\begin{proof}[Proofs  for SPV and SVV]
 We now turn our attention to the stochastic position Verlet (SPV) method, where we use the fact that $(\mathcal{A}\mathcal{V}_{s}\mathcal{A})^{n} = \mathcal{A}\mathcal{V}_{s}(\mathcal{A}\mathcal{V}_{s})^{n-1}\mathcal{A}$. We can instead prove contraction of $\mathcal{A}\mathcal{V}_{s}$. This simplifies the estimation task, but will introduce a prefactor to take into account the estimation of the head and tail operators $\mathcal{A}\mathcal{V}_{s}$ and $\mathcal{A}$.  We denote two synchronous realizations of $\mathcal{A}\mathcal{V}_{s}$ as $\left(x_{j},v_{j}\right)$ and $\left(\Tilde{x}_{j}, \Tilde{v}_{j}\right)$ for $j \in \mathbb{N}$.
Next, write $\overline{x} := \left(\Tilde{x}_{j} - x_{j}\right)$, $\overline{v} = \left(\Tilde{v}_{j} - v_{j}\right)$ and $\overline{z} = \left(\overline{x}, \overline{v}\right)$, where $\overline{z}_{j} = \left(\overline{x}_{j}, \overline{v}_{j}\right)$ for $j = k,k +1$ for $k \in \mathbb{N}$, and $\overline{z}_{k}$ are defined by the update rule
\begin{align*}
    \overline{x}_{k+1} &= \overline{x}_{k} + h \overline{v}_{k}, \quad \overline{v}_{k+1} = \eta \overline{v}_{k} - \frac{1-\eta}{\gamma}Q(\overline{x}_{k} + h\overline{v}_{k}),
\end{align*}
where $\eta = \exp{(-\gamma h)}$ and we define $Q$ from the mean value theorem as $Q = \int^{1}_{t = 0}\nabla^{2}U(\Tilde{x}_{k} + h\Tilde{v}_{k} + t(x_{k} - \Tilde{x}_{k} + h(v_{k} - \Tilde{v}_{k})))dt$,
then $\nabla U\left(\Tilde{x}_{k} + h\Tilde{v}_{k}\right) - \nabla U\left(x_{k} + hv_{k}\right) = Q \left(\overline{x} + h\overline{v}\right)$.
We use the notation of (\ref{eq:contraction_matrix}), where
\begin{align*}
    A &= -c(h)+\frac{2b(1-\eta)}{\gamma}Q - \frac{a(1-\eta)^{2}}{\gamma^{2}}Q^{2},\\
    B &= b(1-\eta) - h - bc(h) + \left(\frac{a\eta(1-\eta)}{\gamma} + \frac{2bh(1-\eta)}{\gamma}\right)Q - \frac{ah(1-\eta)^{2}}{\gamma^{2}}Q^{2},\\
    C &= a(1-\eta^2)-a c(h)-2 b \eta  h-h^2 + \frac{2h(1-\eta)}{\gamma} \left(a \eta + b h\right)Q-\frac{a(1-\eta)^{2}h^{2}Q^2}{\gamma^{2}},
\end{align*}
which motivates the choice of $b = \frac{h}{1-\eta}$. We have chosen $b$, such that $B$ simplifies to 
\begin{align*}
     B &= -bc(h) + \left(\frac{a\eta(1-\eta)}{\gamma} + \frac{2bh(1-\eta)}{\gamma}\right)Q - \frac{ah(1-\eta)^{2}}{\gamma^{2}}Q^{2}.
\end{align*}
It is sufficient to prove that $A \succ 0$ and that $C - BA^{-1}B \succ 0$, noting that $A$, $B$ and $C$ commute as they are all polynomials in $Q$; it is sufficient to prove that $A \succ 0$ and $AC - B^{2} \succ 0$. We will use $P_{A}$ and $P_{AC-B^{2}}$ to denote the eigenvalues of the respective matrices, which will be polynomials in terms of the eigenvalues of $Q$, which we know belong in $[m,M]$. First considering $A$ with our choice of $c(h)$, $a$ and $b$ we have that 
\[
    P_{A}(\lambda) = -\frac{mh}{4\gamma}+\frac{2h}{\gamma}\lambda - \frac{(1-\eta)^{2}}{M\gamma^{2}}\lambda^{2}\geq \frac{h\lambda}{\gamma}\left(\frac{7}{4}- \frac{(1-\eta)^{2}}{h\gamma}\right) > 0,
\]
as $\frac{h\gamma}{2} \leq 1 - \eta \leq h\gamma$ for $h < \frac{1}{2\gamma}$.

Next, $AC - B^2 \succ 0$ is shown as follows
\begin{align*}
    &\frac{P_{AC - B^{2}}(\lambda)}{\lambda} = -\frac{(1-\eta)^{2}\lambda}{\gamma^{2}M^{2}}+h\left(\frac{2(1-\eta^{2})}{\gamma M} + \frac{c(h)}{h}\left(\frac{(1-\eta)^{2}\lambda}{\gamma^{2}M^{2}} -  \frac{1-\eta^{2}}{\lambda M}\right)\right)\\
    &+ h^{2}\left(-\frac{2c(h)}{h\gamma M}(1-\eta^{2}) + \frac{c(h)^{2}}{h^{2}M\lambda} + \frac{\lambda(1-\eta^{2})}{\gamma^{2}M}\right)\\
    &+ h^{3}\left(-\frac{2(1+\eta)}{(1-\eta)\gamma} + \frac{c(h)}{h}\left(\frac{(1 + \eta)}{(1-\eta)\lambda} - \frac{(1-\eta^{2})\lambda}{\gamma^{2}M}\right)\right) + h^{4}\left(\frac{2c(h)(1+\eta)}{\gamma h (1-\eta)} - \frac{c(h)^{2}}{h^{2}(1-\eta)^{2}\lambda}\right)\\
    &\geq -\frac{(1-\eta)h}{\gamma M} + h\left(\frac{2(1-\eta)(1+\eta)}{\gamma M}  - \frac{1-\eta^{2}}{4\gamma M}\right)+  h^{3}\left(-\frac{2(1+\eta)}{(1-\eta)\gamma}\right)\\
    &\geq \frac{3(1-\eta^{2})h}{4\gamma M} - \frac{8(1 - \eta^{2})h}{\gamma^{3}} > 0,
\end{align*}
where we have imposed the restriction $\gamma \geq \sqrt{11M}$. Hence $AC - B^{2} \succ 0$ and our contraction results hold. All computations can be checked using Mathematica.
Now we wish to explicitly compute estimates of the prefactors and the remaining terms, first we consider
\begin{align*}
    &\|\Phi_{\mathcal{V}_{s}} \circ \Phi_{A}(x,v) - \Phi_{\mathcal{V}_{s}} \circ \Phi_{A}(\Tilde{x},\Tilde{v})\|^{2}_{a,b} \leq \left \|(\overline{x} + \frac{h}{2} \overline{v},\eta \overline{v} - \frac{1-\eta}{\gamma}Q(\overline{x} + \frac{h}{2}\overline{v}))\right \|^{2}_{a,b}\\
    &\leq 3\left(\left(1 + 2\left(\frac{1-\eta}{\gamma}\right)^{2}M\right)\|\overline{x}\|^{2} + \left(a\eta^{2} + \frac{h^{2}}{4}\left(1 + 2\left( \frac{1-\eta}{\gamma}\right)^{2}M\right)\right)\|\overline{v}\|^{2}\right)\\
    &\leq 7\|(\overline{x},\overline{v})\|^{2}_{a,b}.
\end{align*}
and
\begin{align*}
    &\|\Phi_{A}(x,v) - \Phi_{A}(\Tilde{x},\Tilde{v})\|^{2}_{a,b} \leq \left \|(\overline{x} + \frac{h}{2} \overline{v},\overline{v})\right \|^{2}_{a,b}\\
    &\leq 7\|(\overline{x},\overline{v})\|^{2}_{a,b}.
\end{align*}
We can combine these estimates to achieve the contraction result for SPV.

Now for the stochastic velocity Verlet we have that  $(\mathcal{V}_{s}\mathcal{A}\mathcal{V}_{s})^{n} = \mathcal{V}_{s}(\mathcal{A}\mathcal{V}_{s}\mathcal{V}_{s})^{n-1}\mathcal{A}\mathcal{V}_{s}$. We will now focus our attention on proving contraction of $\mathcal{A}\mathcal{V}_{s}\mathcal{V}_{s}$. From the fact that $\mathcal{V}_{s}(h/2)\mathcal{V}_{s}(h/2) = \mathcal{V}_{s}(h)$,  contraction of $\mathcal{A}\mathcal{V}_{s}\mathcal{V}_{s}$ is shown in our argument for the stochastic position Verlet.

Now we wish to explicitly compute estimates of the prefactors and the remaining terms, first we consider
\[
    \|\Phi_{\mathcal{V}_{s}}(x,v) - \Phi_{\mathcal{V}_{s}}(\Tilde{x},\Tilde{v})\|^{2}_{a,b} \leq \left \|\left (\overline{x},\eta^{1/2} \overline{v} - \frac{1-\eta^{1/2}}{\gamma}Q\overline{x} \right )\right \|^{2}_{a,b} \leq 6\|(\overline{x},\overline{v})\|^{2}_{a,b}.
\]
and
\begin{align*}
    &\|\Phi_{\mathcal{V}_{s}} \circ \Phi_{A}(x,v) - \Phi_{\mathcal{V}_{s}} \circ \Phi_{A}(\Tilde{x},\Tilde{v})\|^{2}_{a,b} \leq \left \|\left (\overline{x} + h \overline{v},\eta^{1/2} \overline{v} - \frac{1-\eta^{1/2}}{\gamma}Q(\overline{x} + h\overline{v}) \right )\right \|^{2}_{a,b}\\
    &\leq 3\left(\left(1 + 2\left(\frac{1-\eta^{1/2}}{\gamma}\right)^{2}M\right)\|\overline{x}\|^{2} + \left(a\eta + h^{2}\left(1 + 2\left( \frac{1-\eta^{1/2}}{\gamma}\right)^{2}M\right)\right)\|\overline{v}\|^{2}\right)\\
    &\leq 7\|(\overline{x},\overline{v})\|^{2}_{a,b}.
\end{align*}
If we combine these estimates we have the required result.
\end{proof}

\section{Stochastic Gradients}

In many machine learning and statistics applications the cost of a gradient evaluation is high as it requires an evaluation of the entire data set. Instead stochastic gradients are used, where one takes a random sub-sample of the data-set to approximate the gradient with an unbiased estimate. An analysis of convergence rates of the discretizations with stochastic gradients is performed in \cite{monmarche2022hmc}.
\begin{definition} \label{def:stochastic_gradient}
A \textit{stochastic gradient approximation} of a potential $U$ is defined by a function $\mathcal{G}:\mathbb{R}^n \times \Omega \to \mathbb{R}^{n}$ and a probability distribution $\rho$ on a Polish space $\Omega$, satisfying that $\mathcal{G}$ is measurable on $(\Omega,\mathcal{F})$, and that for every $x \in \mathbb{R}^{n}$, for $W\sim \rho$,
\[\mathbb{E}(\mathcal{G}(x,W)) = \nabla U(x).\]
The function $\mathcal{G}$ and the distribution $\rho$  together define the stochastic gradient, which we denote as  $(\mathcal{G}, \rho)$.
\end{definition}

In many applications this can dramatically reduce the computational cost as the approximation will usually come at a a fraction of the workload. The numerical schemes considered in \cite{leimkuhler2023contraction} and in this article are roughly one gradient evaluation per sample, roughly meaning when negating the extra gradient evaluations at the head and tail of the simulation of the algorithm (for the first and last sample). This is done by using the same gradient evaluation in consecutive velocity updates when the position has not been updated, for an increase in computational efficiency. We treat this case when it also comes to stochastic gradients to improve computational efficiency, for example in the BAOAB scheme the last B and first B of each iteration will share an estimate of the force (using the same stochastic gradient evaluation). For clarity a stochastic gradient version of each algorithm is provided in Appendix \ref{Appendix:SG-Algorithms}.

In our convergence rate estimates we impose the assumption that the variance of the Jacobian of the stochastic gradient is bounded.
\begin{assumption}\label{Assumption:Bounded_Variance}
    We assume that the Jacobian of the stochastic gradient $\mathcal{G}$, $D_x\mathcal{G}(x,W)$ exists and it is measurable on $(\Omega,\mathcal{F})$.
    We also assume there exists $C_{G}>0$ such that for $W\sim \rho$,
    \begin{equation*}
        \sup_{x\in \mathbb{R}^{n}}\mathbb{E}\|D_{x}\mathcal{G}(x,W) - \nabla^{2}U(x)\|^{2} \leq C_{G}.
    \end{equation*}
\end{assumption}
Our results extend to the stochastic gradient setting by including a coupling in the mini-batches or the stochastic gradients in the same was as OABAO in \cite{monmarche2022hmc}. Further the results for the other schemes in this article and \cite{leimkuhler2023contraction} are generalized to the case of stochastic gradients when the same stochastic gradient is chosen as in Appendix \ref{Appendix:SG-Algorithms} for each algorithm.  In this way there is still one gradient evaluation per step.

We remark that these assumptions hold when $\mathcal{G}$ is of the form $\mathcal{G}(x,W) = \sum_{i \in W}\nabla U_{i}(x)$, where $W \in \Omega \subset \llbracket 1,N \rrbracket^{n}$ and $(U_{i})_{i \in \llbracket 1,N \rrbracket}$ are strongly convex and gradient-Lipschitz, this is the setting of minibatching in many Bayesian learning problems. The contraction results of Theorem \ref{Theorem:Total_theorem} are extended to the stochastic gradient setting in Theorem \ref{Theorem:Stochastic_Gradients}. We remark that our assumptions are more flexible than the assumptions imposed in \cite{monmarche2022hmc}, where they assume that the stochastic gradient is universally gradient Lipschitz and strongly convex over the entire state space $\Omega$. 

\begin{theorem}\label{Theorem:Stochastic_Gradients}
Consider the numerical schemes for stochastic gradient kinetic Langevin dynamics given in Appendix \ref{Appendix:SG-Algorithms} and Table \ref{Table:stoc_constants}, where the potential $U$ is $m$-strongly convex and $M$-$\nabla$Lipschitz. Assume a stochastic gradient approximation defined by $(\mathcal{G},\rho)$ (see Definition \ref{def:stochastic_gradient}) satisfying  Assumption \ref{Assumption:Bounded_Variance} with constant $C_{G}$. We consider any sequence of synchronously coupled random variables (in Brownian increment and stochastic gradient) with initial conditions $(x_{0},v_{0}) \in \mathbb{R}^{2n}$ and $(\Tilde{x}_{0},\Tilde{v}_{0}) \in \mathbb{R}^{2n}$.

\begin{table}[h]

\begin{center}
\begin{tabular}{ |c|c|c| } 
\hline
{\rm Algorithm} &  c(h) & C(h) \\
\hline
{\rm EM}  &  $mh/2\gamma - 2h^{2}C_{G}/M$ & $1$ \\
{\rm BBK} &  $mh/4\gamma - 4h^{2}C_{G}/M$ & $7 + 3h^{2}C_{G}/M$ \\
{\rm SPV} & $mh/4\gamma - 4h^{2}C_{G}/M$ & $7 + 12h^{2}C_{G}/M$\\
{\rm SVV} & $mh/4\gamma - 4h^{2}C_{G}/M$ & $7 + 6h^{2}C_{G}/M$ \\
{\rm BAOAB} & $h^{2}m/4(1-\eta) - 5h^{2}C_{G}\left(\eta/M + \frac{1}{4}h^{2}\right)$& $7 + 3h^{2}C_{G}/M$ \\
{\rm OBABO} & $h^{2}m/4(1-\eta) - 4h^{2}C_{G}/M$ & $8 + 3h^{2}C_{G}/M$ \\
{\rm rOABAO} & $h^{2}m/4(1-\eta) - 5h^{2}C_{G}\left(\eta/M + \frac{1}{4}h^{2}\right)$ & $8 + 8h^{2}C_{G}/M$ \\ 
{\rm SES/EB} & $mh/4\gamma - 4h^{2}C_{G}/M$ & $1$\\
\hline
\end{tabular}
\end{center}

\caption{Contraction rates $c(h)$ and preconstants $C(h)$ in \eqref{eq:Expected_Contraction}.}
\label{Table:stoc_constants}
\end{table}
Under stepsize restrictions $h < h_{0}$ and $\gamma \geq \gamma_{0}$, where $h_{0}$ and $\gamma_{0}$ are given in Table \ref{Table:constants}, and given initial conditions $(x_{0},v_{0}) \in \mathbb{R}^{2n}$ and $(\Tilde{x}_{0},\Tilde{v}_{0}) \in \mathbb{R}^{2n}$ we have the expected contraction 
\begin{equation}\label{eq:Expected_Contraction}
 \left(\mathbb{E}\|(x_{k} - \Tilde{x}_{k},v_{k} - \Tilde{v}_{k})\|^{2}_{a,b}\right)^{1/2} \leq C(h)(1-c(h))^{s/2}\|(x_{0} - \Tilde{x}_{0},v_{0} - \Tilde{v}_{0})\|_{a,b},   
\end{equation}
with norm given with constants $a = 1/M$ and $b$, where the contraction rate $c(h)$ and the preconstant $C(h)$ are given in Table \ref{Table:stoc_constants} and $b$ and the number of steps $s$ are given in Table \ref{Table:constants} with all parameters specific to each scheme.

\end{theorem}
\begin{remark}
    Compared to Theorem \ref{Theorem:Total_theorem} with deterministic gradients, Theorem \ref{Theorem:Stochastic_Gradients} demonstrates expected contraction, because the randomness from the stochastic gradients can be integrated out. This allows us to make Assumption \ref{Assumption:Bounded_Variance} less restrictive than it would need to be otherwise. Rather than deterministic contraction we have contraction in expectation.
\end{remark}
\begin{remark}
Our analysis suggests a reduction in the convergence rate for large gradient noises, which we have observed in numerical experiments when using sub-sampling and very small batches. For large gradient noise $C_{G}$ and stepsize $h$ it is possible that these bounds become vacuous and the loss of convergence was also confirmed in our experiments.
\end{remark}

\begin{remark}
    The implementation of the BAOAB algorithm and other algorithms considered in Section \ref{Appendix:SG-Algorithms} is non-Markovian, because the last $B$ step of each iteration and the first $B$ step of the next iteration share the same stochastic gradient sample. This is not an issue in our convergence rate framework as we consider convergence of a different operator, which is Markovian, for example $\mathcal{ABAO}$ for BAOAB, which does not share stochastic gradients with consecutive iterations. We simplify the problem into proving convergence of an operator which only has a single gradient evaluation and hence is Markovian in the stochastic gradient setting.
\end{remark}

\begin{proof}[Proof of Theorem \ref{Theorem:Stochastic_Gradients}]
For stochastic gradients we synchronously couple Brownian increments as well as the stochastic gradients. We wish to instead consider expected contraction of the update rule we used to prove contraction in the full gradient setting, i.e. for synchronously coupled (in stochastic gradient and Brownian increment) iterates $(x_{l},v_{l}),(\Tilde{x}_{l},\Tilde{v}_{l}) \in \mathbb{R}^{2n}$ for $l \in \mathbb{N}$ and $(\overline{x}_{l},\overline{v}_{l}) = (\Tilde{x}_{l},\Tilde{v}_{l}) - (x_{l},v_{l})$ and for $k \in \mathbb{N}$
\[
\mathbb{E}\|(\overline{x}_{k+1},\overline{v}_{k+1})\|^{2}_{a,b} \leq \left(1-c(h)\right)\|(\overline{x}_{k},\overline{v}_{k})\|^{2}_{a,b}
\]
then we have
\[
\mathbb{E}\left(\overline{z}^{T}_{k}P^{T}MP\overline{z}_{k}\right) \leq \left(1-c(h)\right)\overline{z}^{T}_{k}M\overline{z}_{k}.
\]
Now if $\Tilde{Q}$ is defined through the mean value theorem of $D_{x}\mathcal{G}$ (the Jacobian of $\mathcal{G}$) and is a random variable in $W$, such that $\mathbb{E}(\Tilde{Q}) = Q$. Then $P^{T}MP$ is of the form
\begin{equation*}
\mathcal{P}(\Tilde{Q}) = \begin{pmatrix}
    P_{1}(\Tilde{Q}) & P_{2}(\Tilde{Q})  \\
    P_{2}(\Tilde{Q})  & P_{3}(\Tilde{Q}) 
\end{pmatrix},    
\end{equation*}
where $P_{1}, P_{2}$ and $P_{3}$ are quadratics in $\Tilde{Q}$ of the form
\begin{align*}
    P_{1}(\Tilde{Q}) &= a_{0} + a_{1}\Tilde{Q} + a_{2}\Tilde{Q}^{2},\\
    P_{2}(\Tilde{Q}) &= b_{0} + b_{1}\Tilde{Q} + b_{2}\Tilde{Q}^{2},\\
    P_{3}(\Tilde{Q}) &= c_{0} + c_{1}\Tilde{Q} + c_{2}\Tilde{Q}^{2}.
\end{align*} 
Then we have 
\begin{equation*}
\mathbb{E}(P^{T}MP) =  \mathcal{P}(Q) + \begin{pmatrix}
    a_{2}\mathbb{E}(\Tilde{Q} - Q)^{2} & b_{2}\mathbb{E}(\Tilde{Q} - Q)^{2}  \\
    b_{2}\mathbb{E}(\Tilde{Q} - Q)^{2}  & c_{2}\mathbb{E}(\Tilde{Q} - Q)^{2} 
\end{pmatrix},    
\end{equation*}
 in combination with the Theorem \ref{Theorem:Total_theorem} result we have that 
\begin{align*}
  \mathbb{E}\|(x_{k+1},v_{k+1})\|^{2}_{a,b} &\leq \left(1-c(h)\right)\|(x_{k},v_{k})\|^{2}_{a,b} + z_{k}^{T}\begin{pmatrix}
    a_{2}\mathbb{E}(\Tilde{Q} - Q)^{2} & b_{2}\mathbb{E}(\Tilde{Q} - Q)^{2}  \\
    b_{2}\mathbb{E}(\Tilde{Q} - Q)^{2}  & c_{2}\mathbb{E}(\Tilde{Q} - Q)^{2} 
\end{pmatrix}z_{k}\\
&= \left(1-c(h)\right)\|(x_{k},v_{k})\|^{2}_{a,b} + z_{k}^{T}\mathcal{R}(\Tilde{Q})z_{k},
\end{align*}
where we use the notation
\[
\mathcal{R}(\Tilde{Q}) := \begin{pmatrix}
    a_{2}\mathbb{E}(\Tilde{Q} - Q)^{2} & b_{2}\mathbb{E}(\Tilde{Q} - Q)^{2}  \\
    b_{2}\mathbb{E}(\Tilde{Q} - Q)^{2}  & c_{2}\mathbb{E}(\Tilde{Q} - Q)^{2} 
\end{pmatrix}.
\]
Then we will bound the remainder term $z^{T}\mathcal{R}(\Tilde{Q})z$ separately for each scheme and we refer the reader to the contraction estimate proofs in \cite{leimkuhler2023contraction} for the coefficients $a_{2}, b_{2}$ and $c_{2}$ for Euler-Maruyama, BAOAB, OBABO and SES and to Section \ref{Sec:Proofs} for the schemes analyzed in this article. We remark that we analyze the update rules for which we proved contraction for all the schemes, which aren't necessarily the same as the scheme for example we analyze $\mathcal{ABAO}$ for BAOAB. Throughout these estimates we use the equivalence of norms in Section \ref{Sec:Quadratic_Norm} and the stepsize and parameter restrictions imposed in the contraction estimates of the respective schemes. We define $\textnormal{Var}(\Tilde{Q}) := \mathbb{E}(\Tilde{Q} - Q)^{2}$, to be the variance of $\Tilde{Q}$.
\begin{enumerate}
    \item For the Euler-Maruyama $a_{2} > 0$ and $b_{2} = c_{2} = 0$, therefore we have for $z = (x,v) \in \mathbb{R}^{2n}$
\begin{align*}
  z^{T}\mathcal{R}(\Tilde{Q})z &\leq h^{2}aC_{G}\|x\|^{2}\\
  &\leq  2h^{2}aC_{G}\|z\|^{2}_{a,b}.
\end{align*}
    \item For BAOAB we have for $\mathcal{ABAO}$
    \begin{align*}
      z^{T}\mathcal{R}(\Tilde{Q})z &= ah^{2}\left(\eta^{2} + b\eta h M + \frac{h^{2}}{4}M\right)\left( x + \frac{h}{2}v\right)^{T}\textnormal{Var}(\Tilde{Q})\left(x + \frac{h}{2}v\right)\\
        &\leq 4ah^{2}C_{G}\left(\eta^{2} + b\eta h M + \frac{h^{2}}{4}M\right)\|(x,v)\|^{2}_{a,b}\\
        &\leq 5ah^{2}C_{G}\left(\eta + \frac{h^{2}}{4}M\right)\|(x,v)\|^{2}_{a,b}.
\end{align*}

\item For OBABO we have for $\mathcal{ABOB}$
    \begin{align*}
      z^{T}\mathcal{R}(\Tilde{Q})z &= a\left(\eta + 1 \right)^{2}\frac{h^{2}}{4}\left( \left( x + hv\right)^{T}\textnormal{Var}(\Tilde{Q})\left(x + hv\right)\right)\\
      &\leq a\left(\eta + 1\right)^{2}h^{2}C_{G}\|(x,v)\|^{2}_{a,b} \leq 4ah^{2}C_{G}\|(x,v)\|^{2}_{a,b}.
\end{align*}

\item For SES we have 
    \begin{align*}
      z^{T}\mathcal{R}(\Tilde{Q})z &= \left( \frac{a\left(1-\eta\right)^{2}}{\gamma^{2}} + \frac{2b\left(1-\eta\right)\left(-1 + \eta + \gamma h\right)}{\gamma^{3}} + \frac{\left(-1 + \eta +\gamma h \right)^{2}}{\gamma^{4}}\right)x^{T}\textnormal{Var}(\Tilde{Q})x\\
      &\leq 4ah^{2}C_{G}\|(x,v)\|^{2}_{a,b}.
\end{align*}

\item For rOABAO we have for $r\mathcal{ABAO}$
    \begin{align*}
      z^{T}\mathcal{R}(\Tilde{Q})z &= \left(a\eta^{2}h^{2}+ b\eta h^{3} + \frac{1}{4}h^{4}\right)\left(\left(x + uv\right)^{T}\textnormal{Var}(\Tilde{Q})\left(x + uv\right)\right)\\
      &\leq 4\left(a\eta^{2}h^{2}+ b\eta h^{3} + \frac{1}{4}h^{4}\right)C_{G}\|(x,v)\|^{2}_{a,b}\\
      &\leq 5h^{2}\left(a\eta + \frac{1}{4}h^{2}\right)C_{G}\|(x,v)\|^{2}_{a,b}.
\end{align*}

\item For BBK we have for $\mathcal{A}\mathcal{B}_{2}\mathcal{B}_{1}$
    \begin{align*}
      z^{T}\mathcal{R}(\Tilde{Q})z &= \frac{ah^{2}}{\left(\frac{1}{2}\gamma h + 1\right)^{2}}\left(x + hv\right)^{T}\textnormal{Var}(\Tilde{Q})\left(x + hv\right)\\
      &\leq 4ah^{2}C_{G}\|(x,v)\|^{2}_{a,b}.
\end{align*}

\item For SPV and SVV we have for $\mathcal{AB}_{s}$
    \begin{align*}
      z^{T}\mathcal{R}(\Tilde{Q})z &= \frac{a(1-\eta)^{2}}{\gamma^{2}}\left(x + hv\right)^{T}\textnormal{Var}(\Tilde{Q})\left(x + hv\right)\\
      &\leq 4ah^{2}C_{G}\|(x,v)\|^{2}_{a,b}.
    \end{align*}

\end{enumerate}

Due to the stochastic gradient we have to recompute the preconstants for BAOAB, OBABO, rOABAO, BBK, SPV and SVV. We will use the fact that
\begin{align*}
    \mathbb{E}\|\Tilde{Q}x\|^{2} &= \mathbb{E}\|(\Tilde{Q} - Q)x\|^{2} + \|Qx\|^{2}\\
    &\leq \left(C_{G} + M^{2}\right)\|x\|^{2}
\end{align*}
to compute the preconstants, we will now perform the calculation for all necessary schemes, we refer you to \cite{leimkuhler2023contraction} and Section \ref{Sec:Proofs} for the operators at the head and tail of the contraction estimates that we will need to bound. 

\begin{enumerate}
    \item For BAOAB in \cite{leimkuhler2023contraction}, contraction of $\mathcal{ABAO}$ was proven and hence we need to bound $\mathcal{BAO}$ and $\mathcal{AB}$. For $k \in \mathbb{N}$ and iteration $(x_{k},v_{k}),(\Tilde{x}_{k},\Tilde{v}_{k}) \in \mathbb{R}^{2n}$ we first estimate
\begin{align*}
    &\mathbb{E}\|\Phi_{\mathcal{AB}}\left(\Tilde{x}_{k},\Tilde{v}_{k}\right)-\Phi_{\mathcal{AB}}\left(x_{k},v_{k}\right)\|^{2}_{a,b}\\
    &\leq 3\left(\left(1 + \frac{ah^{2}}{2}\left(C_{G} + M^{2}\right)\right)\|\overline{x}_{k}\|^{2} + \left(a + \frac{h^{2}}{4} + \frac{ah^{4}}{8}\left(M^{2} + C_{G}\right)\right)\|\overline{v}_{k}\|^{2}\right)\\
    &\leq 7\|\left(\overline{x}_{k},\overline{v}_{k}\right)\|^{2}_{a,b} + 3\left(\frac{1}{2}ah^{2}C_{G}\|\overline{x}_{k}\|^{2} + a\frac{h^{4}}{8}C_{G}\|\overline{v}_{k}\|^{2}\right)\\
    &\leq \left(7 + 3ah^{2}C_{G}\right)\|(\overline{x}_{k},\overline{v}_{k})\|^{2}_{a,b},
\end{align*}
and where we have defined $\overline{x}_{k} = \Tilde{x}_{k} - x_{k}$ and $\overline{v}_{k} = \Tilde{v}_{k} - v_{k}$. We then estimate
\begin{align*}
    &\mathbb{E}\|\Phi_{\mathcal{BAO}}\left(\Tilde{x}_{k},\Tilde{v}_{k}\right)-\Phi_{\mathcal{BAO}}\left(x_{k},v_{k}\right)\|^{2}_{a,b}\\
    &\leq 3\left(\left(1 +\frac{h^{2}}{4}\left(a\eta^{2} + \frac{h^{2}}{2}\right)\left(C_{G} + M^{2}\right)\right)\|\overline{x}_{k}\|^{2} + \left(a\eta^{2} + \frac{h^{2}}{2}\right)\|\overline{v}_{k}\|^{2}\right)\\
    &\leq \left(7 + \frac{h^{2}}{2}C_{G}\left(a\eta^{2} + \frac{h^{2}}
{2}\right)\right)\|(\overline{x}_{k},\overline{v}_{k})\|^{2}_{a,b}.
\end{align*}
Now combining these constants we have the relevant preconstant for the constraction estimate for BAOAB.
\item Similarly we have the following bounds to compute preconstants for OBABO:
\begin{align*}
    &\mathbb{E}\|\Phi_{\mathcal{ABO}}\left(\Tilde{x}_{k},\Tilde{v}_{k}\right)-\Phi_{\mathcal{ABO}}\left(x_{k},v_{k}\right)\|^{2}_{a,b}\\
    &\leq 3\left(\left(1 + \frac{ah^{2}}{2}\left(C_{G} + M^{2}\right)\right)\|\overline{x}_{k}\|^{2}+ \left(a + h^{2} + \frac{ah^{4}}{2}\left(C_{G} + M^{2}\right)\right)\|\overline{v}_{k}\|^{2}\right)\\
    &\leq \left(8 + 3ah^{2}C_{G}\right)\|(\overline{x}_{k},\overline{v}_{k})\|^{2}_{a,b},
\end{align*}
and then we estimate
\begin{align*}
    &\mathbb{E}\|\Phi_{\mathcal{OB}}\left(\Tilde{x}_{k},\Tilde{v}_{k}\right)-\Phi_{\mathcal{OB}}\left(x_{k},v_{k}\right)\|^{2}_{a,b}\\
    &\leq 3\left(\left(\frac{1}{2} + \frac{ah^{2}}{4}\left(C_{G} + M^{2}\right)\right)\|\overline{x}_{k}\|^{2} + a\|\overline{v}_{k}\|^{2}\right)\\
    &\leq \left(6 + 2ah^{2}C_{G}\right)\|(\overline{x}_{k},\overline{v}_{k})\|^{2}_{a,b}.
\end{align*}
Now combining these constants we have the relevant preconstant for the contraction estimate for OBABO.
    \item Now considering rOABAO the preconstant bound for $\mathcal{O}$ is unaffected so we just need to estimate the operator $r\mathcal{ABAO}$. We have
\begin{align*}
    &\mathbb{E}\|\Phi_{r\mathcal{ABAO}}\left(\Tilde{x}_{k},\Tilde{v}_{k}\right)-\Phi_{r\mathcal{ABAO}}\left(x_{k},v_{k}\right)\|^{2}_{a,b} \leq \left(14 + 14ah^{2}C_{G}\right)\|(\overline{x}_{k},\overline{v}_{k})\|^{2}_{a,b},
\end{align*}
and we can combine this wiht the bound for $\mathcal{O}$ given in Section \ref{Sec:Proofs}.
    \item For BBK we have the estimates
\begin{align*}
    &\mathbb{E}\|\Phi_{B_{2}}\circ \Phi_{A}\left(\Tilde{x}_{k},\Tilde{v}_{k}\right)-\Phi_{B_{2}}\circ \Phi_{A}\left(x_{k},v_{k}\right)\|^{2}_{a,b} \leq \left(7 + 3ah^{2}C_{G}\right)\|(\overline{x}_{k},\overline{v}_{k})\|^{2}_{a,b},
\end{align*}
and
\begin{align*}
    &\mathbb{E}\|\Phi_{B_{1}}\left(\Tilde{x}_{k},\Tilde{v}_{k}\right)-\Phi_{B_{1}}\left(x_{k},v_{k}\right)\|^{2}_{a,b} \leq \left(6 + 2ah^{2}C_{G}\right)\|(\overline{x}_{k},\overline{v}_{k})\|^{2}_{a,b},
\end{align*}
which we can combine to get the desired preconstant. 
    \item For SPV we estimate
\begin{align*}
    &\mathbb{E}\|\Phi_{\mathcal{V}_{s}}\circ \Phi_{A}\left(\Tilde{x}_{k},\Tilde{v}_{k}\right)-\Phi_{\mathcal{V}_{s}}\circ \Phi_{A}\left(x_{k},v_{k}\right)\|^{2}_{a,b} \leq \left(7 + 12ah^{2}C_{G}\right)\|(\overline{x}_{k},\overline{v}_{k})\|^{2}_{a,b},
\end{align*}
and we can use the previous estimate of $\mathcal{A}$ to estimate the preconstant. 
\item Finally for SVV we estimate
\begin{align*}
    &\mathbb{E}\|\Phi_{\mathcal{V}_{s}}\left(\Tilde{x}_{k},\Tilde{v}_{k}\right)-\Phi_{\mathcal{V}_{s}}\left(x_{k},v_{k}\right)\|^{2}_{a,b} \leq \left(6 + 3ah^{2}C_{G}\right)\|(\overline{x}_{k},\overline{v}_{k})\|^{2}_{a,b},
\end{align*}
and
\begin{align*}
    &\mathbb{E}\|\Phi_{\mathcal{V}_{s}}\circ \Phi_{A} \left(\Tilde{x}_{k},\Tilde{v}_{k}\right)-\Phi_{\mathcal{V}_{s}}\circ \Phi_{A} \left(x_{k},v_{k}\right)\|^{2}_{a,b} \leq \left(7 + 6ah^{2}C_{G}\right)\|(\overline{x}_{k},\overline{v}_{k})\|^{2}_{a,b}.
\end{align*}

\end{enumerate}
We have all desired preconstants and penalty terms for the contraction rate when the gradient is a stochastic estimate.

\end{proof}

\begin{proposition}\label{prop:Wasserstein_stoch}
Consider the numerical schemes for stochastic gradient kinetic Langevin dynamics given in Appendix \ref{Appendix:SG-Algorithms} and Table \ref{Table:stoc_constants}, where the potential $U$ is $m$-strongly convex and $M$-$\nabla$Lipschitz. Assume a stochastic gradient approximation defined by $(\mathcal{G},\rho)$ (see Definition \ref{def:stochastic_gradient}) satisfying  Assumption \ref{Assumption:Bounded_Variance} with constant $C_{G}$. We use $P^{n}_{h}$ to denote the marginal transition kernel of the numerical schemes. For the constants given in Table \ref{Table:constants} and \ref{Table:stoc_constants} we have for any two synchronously coupled chains, $\left(x_{k},v_{k}\right)$ and $\left(\Tilde{x}_{k},\Tilde{v}_{k}\right)$ under the assumptions specific to the schemes imposed in Theorem \ref{Theorem:Stochastic_Gradients} we have for all $1 \leq p \leq \infty$ and all $\mu,\nu \in \mathcal{P}_{p}(\mathbb{R}^{2n})$, and  all $k \in \mathbb{N}$,
\[
\mathcal{W}^{2}_{2}\left(\nu P^k_{h} ,\mu P^k_{h} \right) \leq 3C(h)\max{\left\{M,\frac{1}{M}\right\}}\left(1 - c\left(h\right)\right)^{k} \mathcal{W}^{2}_{2}\left(\nu,\mu\right).
\]
\end{proposition}
\begin{proof}
   We remark that the stochastic gradients are independent from position and hence can be marginalized out in the following estimates over the extended state space. We first denote $\hat{P}_{h}$ to be the transition kernel for which contraction is proved in Theorem \ref{Theorem:Total_theorem} or \cite{leimkuhler2023contraction}. For example stochastic gradient $\mathcal{ABAO}$ for BAOAB. From Theorem \ref{Theorem:Stochastic_Gradients} and following \cite{monmarche2021high}[Corollary 20] we know that for $z_{k} = (x_{k},v_{k})$, $\Tilde{z}_{k} = (\Tilde{x}_{k},\Tilde{v}_{k})$ such that $z_{0} = (x_{0},v_{0}) \sim \mu$ and $\Tilde{z}_{0} = (\Tilde{x}_{0},\Tilde{v}_{0}) \sim \nu$ and $(z_{0},\Tilde{z}_{0})$ is a $\mathcal{W}_{2}$ optimal coupling of $\mu$ and $\nu$ then under $\hat{P}_{h}$
    \[
    \mathcal{W}^{2}_{2,a,b}\left(\mu \hat{P}^{k}_{h},\nu \hat{P}^{k}_{h}\right) \leq \mathbb{E}\|z_{k} - \Tilde{z}_{k}\|^{2}_{a,b} \leq (1-c(h))^{k}\mathcal{W}^{2}_{2,a,b}\left(\mu,\nu\right),
    \]
    then we can use the equivalence of norms in Section \ref{Sec:Quadratic_Norm} and the preconstant estimates of Table \ref{Table:stoc_constants} to achieve the desired result for $P_{h}$.
\end{proof}

\begin{remark}
    We remark that the contraction rate of BAOAB and rOABAO can be upper bounded by a simpler form for example $\mathcal{O}\left(mh^{2}/(1-\eta) - h^{2}C_{G}/M\right)$, but we have included the more detailed estimate because it has the property that as you take the friction parameter $\gamma \to \infty$ then the contraction rate is of the same order as the overdamped Langevin dynamics scheme as discussed in Section \ref{sec:Overdamped_Langevin}.
\end{remark}

\subsection{Overdamped Langevin Dynamics}\label{sec:Overdamped_Langevin}

If we take the limit  of the friction parameter ($\gamma \to \infty$) in (\ref{eq:underdamped_langevin}), with a time-rescaling ($t' =\gamma t$) we have the overdamped Langevin dynamics given by (see \cite{pavliotis2014stochastic}[Sec 6.5])
\begin{equation}\label{eq:overdamped}
dX_{t}  = -\nabla U(X_t) dt + \sqrt{2}dW_t.
\end{equation}
This equation has invariant measure with density proportional to $\exp(-U(x))$, like the marginal of the underdamped counterpart.

Two discretizations of the SDE (\ref{eq:overdamped}) considered in \cite{leimkuhler2023contraction} and linked to the BAOAB and OBABO schemes through high friction limits are the Euler-Maruyama (EM) method which is defined by the update rule
\begin{equation}
 x_{k+1} = x_{k} - h\nabla U\left(x_{k}\right) + \sqrt{2h}\xi_{k+1},  
\end{equation}
and the BAOAB limit method of Leimkuhler and Matthews (LM)(\cite{leimkuhler2013rational,leimkuhler2014long}) which is defined by the update rule
\[
x_{k+1} = x_{k} - h\nabla U\left(x_{k}\right) + \sqrt{2h}\frac{\xi_{k+1} + \xi_{k}}{2}.
\]

Now we can apply coupling arguments to the overdamped dynamics in a simpler way (using the standard Euclidean distance), but we will also consider the case of stochastic gradients. Coupling arguments in the overdamped setting have been extensively studied with and without stochastic gradients (see \cite{dalalyan2017theoretical,durmus2017nonasymptotic,cheng2018convergence,dalalyan2017further,durmus2019high, durmus2019analysis, dwivedi2018log}).

\begin{proposition}\label{prop:Overdamped_stoch}
    Consider the Euler-Maruyama and Leimkuhler-Matthews schemes for stochastic gradient Langevin dynamics, where the potential $U$ is $m$-strongly convex and $M$-$\nabla$Lipschitz. Assume a stochastic gradient approximation defined by $(\mathcal{G},\rho)$ (see Definition \ref{def:stochastic_gradient}) satisfying  Assumption \ref{Assumption:Bounded_Variance} with constant $C_{G}$. We consider any sequence of synchronously coupled random variables (in Brownian increment and stochastic gradient) with initial conditions $x_{0} \in \mathbb{R}^{n}$ and $y_{0} \in \mathbb{R}^{n}$. We have the contraction property
\[
\mathbb{E}\|x_{k} - y_{k}\|^{2} \leq (1-\left(hm\left(2-hM\right) - h^{2}C_{G}\right))^{k}\|x_{0} - y_{0}\|^{2}.
\]
\end{proposition}

\begin{proof}
If we first consider two chains $x_{k}$ and $y_{k}$ with shared noise such that
\begin{align*}
    x_{k+1} = x_{k} - h \mathcal{G}(x_{k},W_{k}) + \sqrt{2h}\xi_{k}, \quad y_{k+1} = y_{k} - h\mathcal{G}(y_{k},W_{k}) + \sqrt{2h}\xi_{k},
\end{align*}
where $W_{k} \sim \rho$ and $\xi_{k} \sim \mathcal{N}(0,1)^{n}$ for all $k \in \mathbb{N}$ and this can be either the Euler-Maruyama or Leimkuhler-Matthews method, we have chosen Euler-Maruyama. Then we have that 
\begin{align*}
    &\mathbb{E}\|x_{k+1} - y_{k+1}\|^2 = \mathbb{E}\|x_{k} - y_{k} + h(-\mathcal{G}(x_{k},W_{k}) - (-\mathcal{G}(y_{k},W_{k}))\|^{2} \\
    &= \|x_{k} - y_{k}\|^{2} - 2h \langle x_{k} - y_{k}, Q(x_{k} - y_{k})\rangle \\
    & \hspace{0.3in} + h^{2}\langle x_{k} - y_{k}, Q^{2} (x_{k} - y_{k}) \rangle + h^{2}\langle x_{k} - y_{k},\textnormal{Var}(\Tilde{Q})\left(x_{k} - y_{k}\right)\rangle,
\end{align*}
where $\Tilde{Q} = \int^{1}_{t = 0}D_{x}\mathcal{G}(x_{k} + t(y_{k} - x_{k}),W_{k})dt$ and $Q = \mathbb{E}(\Tilde{Q})$. $Q$ has eigenvalues which are bounded between $m$ and $M$, so $Q^{2} \preceq MQ$, and hence 
\[
h^{2}\langle x_{k} - y_{k}, Q^{2} (x_{k} - y_{k}) \rangle \leq h^{2}M \langle x_{k} - y_{k}, Q(x_{k} - y_{k}) \rangle.
\]
If we impose Assumption \ref{Assumption:Bounded_Variance} on the Jacobian of the stochastic gradient then we have that 
\[
h^{2}\langle x_{k} - y_{k},\textnormal{Var}(\Tilde{Q})\left(x_{k} - y_{k}\right)\rangle \leq h^{2}C_{G}\|x_{k} - y_{k}\|^{2},
\]
and therefore 
\begin{align*}
    \|x_{k+1} - y_{k+1}\|^2 &\leq \|x_{k} - y_{k}\|^{2} - h(2 - hM)\langle x_{k} - y_{k},Q(x_{k} - y_{k}) \rangle \\
    & \hspace{0.3in} + h^{2}\langle x_{k} - y_{k},\textnormal{Var}(\Tilde{Q})\left(x_{k} - y_{k}\right)\rangle \\
    &\leq \|x_{k} - y_{k}\|^{2}(1 - hm(2-hM) + h^{2}C_{G}).
\end{align*}
\end{proof}
In the same way as Proposition \ref{prop:Wasserstein_stoch} the contraction property of Proposition \ref{prop:Overdamped_stoch} implies convergence in Wasserstein distance in the stochastic gradient setting. Now if we take the limit as $\gamma \to \infty$ for the BAOAB and rOABAO scheme we get a contribution from the stochastic gradient of $\mathcal{O}\left(h^{4}C_{G}\right)$ in the convergence rate estimate, and for the overdamped analysis we have a contribution of $\mathcal{O}\left(h^{4}C_{G}\right)$ in the high friction limit of BAOAB and rOABAO. However for OBABO we get a contribution of $\mathcal{O}\left(h^{2}C_{G}/M\right)$, which agrees with the overdamped Langevin analysis for the largest choice of stepsize. 

\section{Overdamped Limit} 
Now reflecting on the contraction rates achieved in \cite{leimkuhler2023contraction} we also consider the $\gamma$-limit convergent  (GLC) property, i.e. the convergence of the integrator obtained in the $\gamma \to \infty$ limit.

\subsection{BAOAB and OBABO}
As originally noted in \cite{leimkuhler2013rational} the high friction limit of the BAOAB method is
\begin{align*}
    x_{k+1} &= x_{k} - \frac{h^{2}}{2} \nabla U(x_{k}) + \frac{h}{2}\left(\xi_{k} + \xi_{k+1}\right),
\end{align*}
which is simply the Leimkuhler-Matthews scheme of \cite{leimkuhler2013rational} with stepsize $h^{2}/2$. As studied in \cite{leimkuhler2023contraction} this imposes stepsize restrictions $h^{2} \leq 2/M$ due to the analysis of the overdamped counterpart.  The limiting contraction rate is given by
\[
\lim_{\gamma \to \infty} c\left(h\right) = \frac{h^{2}m}{4},
\]
which agrees with contraction rate estimates for the overdamped dynamics. 

Similarly for the OBABO scheme the limiting method
\begin{align*}
    x_{k+1} &= x_{k} -\frac{h^{2}}{2}\nabla U(x_{k}) + h \xi_{k+1},
\end{align*}
is the Euler-Maruyama scheme for overdamped Langevin with stepsize $h^{2}/2$, and the limit contraction rate is consistent with the contraction rate for the underdamped dynamics as noted in \cite{leimkuhler2023contraction}.

\subsection{BBK Integrator}
Taking the limit as $\gamma \to \infty$ and by considering two consecutive iterations ($v_{k+1} = - v_{k}$ in this limit) one arrives at the following update rule
\[
x_{k+2} = x_{k} - \frac{h^{2}}{2}\left(\nabla U(x_{k+1}) + \nabla U(x_{k}) \right) + \frac{\sqrt{2\gamma}h^{3/2}}{2} \left(\xi_{k} + \xi_{k+1} \right),
\]
and hence the method is not GLC as this does not converge to overdamped dynamics as the stepsize is taken to zero.

\subsection{Stochastic position Verlet and  stochastic velocity Verlet}
If one takes the limit as $\gamma \to \infty$ for the stochastic position and velocity Verlet then we get the operators
\begin{align*}
    \mathcal{V}_{s}(h)&: v \to \xi,\\
    \mathcal{A}(h)&: x \to x + h v,
\end{align*}
hence these schemes do not converge to the overdamped dynamics  as one takes the stepsize to zero.
\subsection{rOABAO}

The rOABAO scheme is GLC and, interestingly, by taking the high friction limit one arrives at the scheme
\[
x_{k+1} = x_{k} - \frac{h^{2}}{2}\nabla U(x_{k} + u \xi_{k}) + h \xi_{k},
\]
where $u \sim \mathcal{U}(0,h)$, which has the correct invariant measure and is a randomized midpoint version of the Euler-Maruyama scheme for overdamped Langevin dynamics and the one-step HMC scheme of \cite{bou2022unadjusted}.
\section{Numerical Experiments}

To quantify and validate our convergence results and contraction rates we approximate the spectral gap of the numerical scheme $c(h)$ for an Anisotropic Gaussian example. We then compare this to the continuous dynamics via
\[
\frac{1-c(h)}{h},
\]
which converges to the spectral gap of the continuous dynamics as $h \to 0$, and is normalized by stepsize. We also compare the bias of the numerical integrators in a Bayesian classification application.

\subsection{Anisotropic Gaussian}

We first consider a simple low-dimensional example to compare the convergence rates,  the anisotropic Gaussian distribution on $\mathbb{R}^{2}$ with potential $U: \mathbb{R}^{2} \mapsto \mathbb{R}$ given by $U(x,y) = \frac{1}{2}mx^{2} + \frac{1}{2}My^{2}$. This potential satisfies Assumption \ref{sec:assumptions} with constants $M$ and $m$ respectively. For this example we can analytically solve for the contraction rates, which coincide with the convergence rates of $\mathbb{E}(X_{n})$. We can do this by computing the spectral gap of the transition matrix $P$, by $1 - |\lambda_{\rm{max}}|$, where $\lambda_{\rm{max}}$ is the largest eigenvalue of the matrix $P$ due to Gelfand's formula. This converges to the spectral gap of the continuous dynamics as $h \to 0$.

The dependence of the convergence rate on the friction parameter $\gamma$ is given in Figure \ref{fig:True_Contraction}. We will study how this changes for the discretisations with contour plots of stepsize versus contraction rate for all the numerical methods we consider. If we take a slice of our contour plots for small stepsizes then this will coincide with Figure \ref{fig:True_Contraction}. This is given in Figure \ref{fig:fig}.

\begin{figure}[H]
\begin{centering}
\includegraphics[width=7cm]{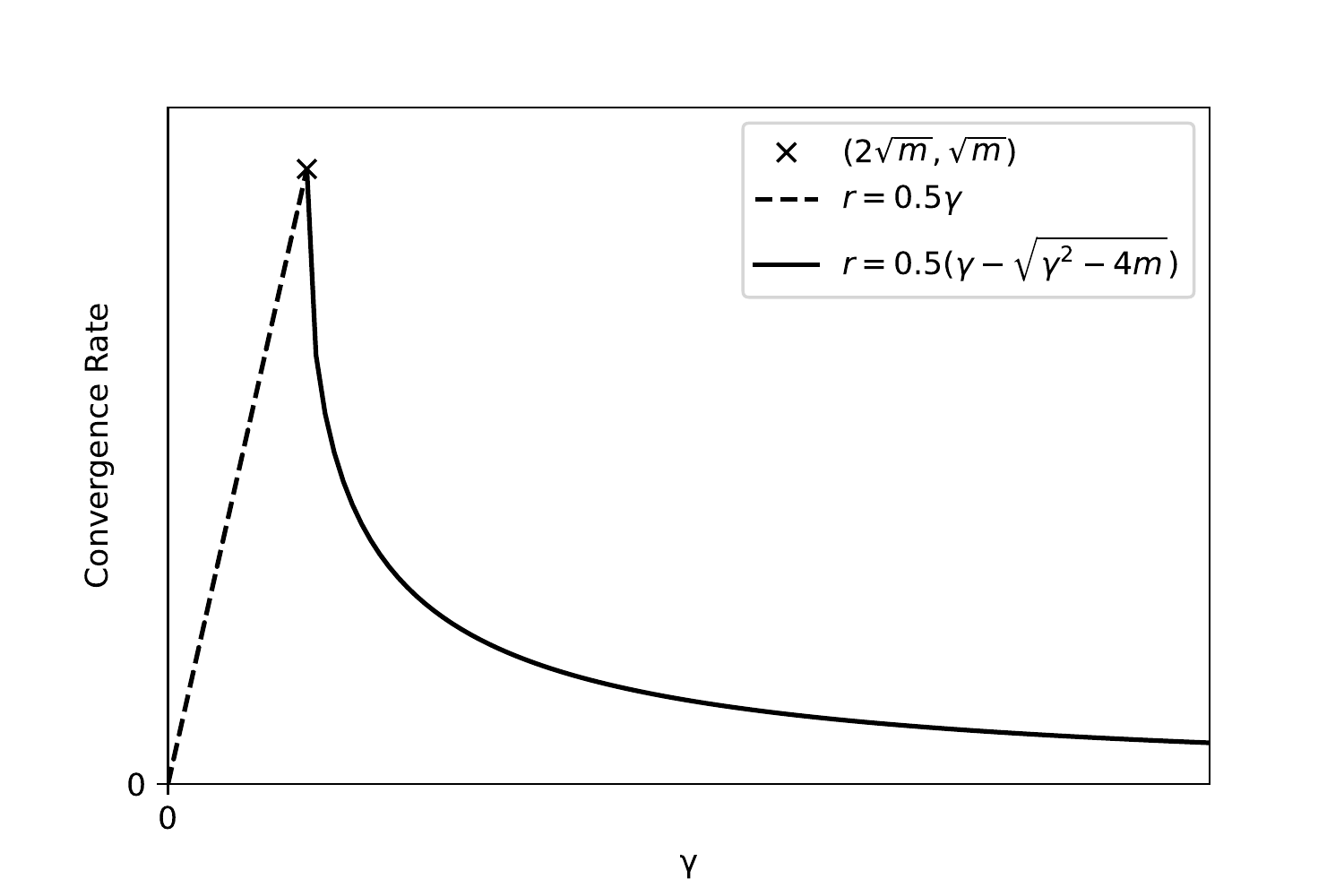}
\par\end{centering}
\caption{\label{fig:True_Contraction}}
\end{figure}

Due to the fact that each update matrix $P$ for the anisotropic Gaussian using the rOABAO scheme is in fact a random matrix, we estimate the contraction rate using \cite{furstenberg1960products,kargin2010products}, where
\[
\lim_{N \to \infty} \log{\|P_{1}P_{2}...P_{N}\|/N} \to \log(1-c(h)),
\]
where $P_{i}$ for $i \in \mathbb{N}$ is the transition matrix of the $i^{\rm{th}}$ iteration. We approximate this limit by Monte Carlo simulations with a random $u \sim [0,h]$ from the randomized midpoint at each stage to approximate the spectral radius.

\begin{figure}[htb]
\includegraphics[width=5.0in]{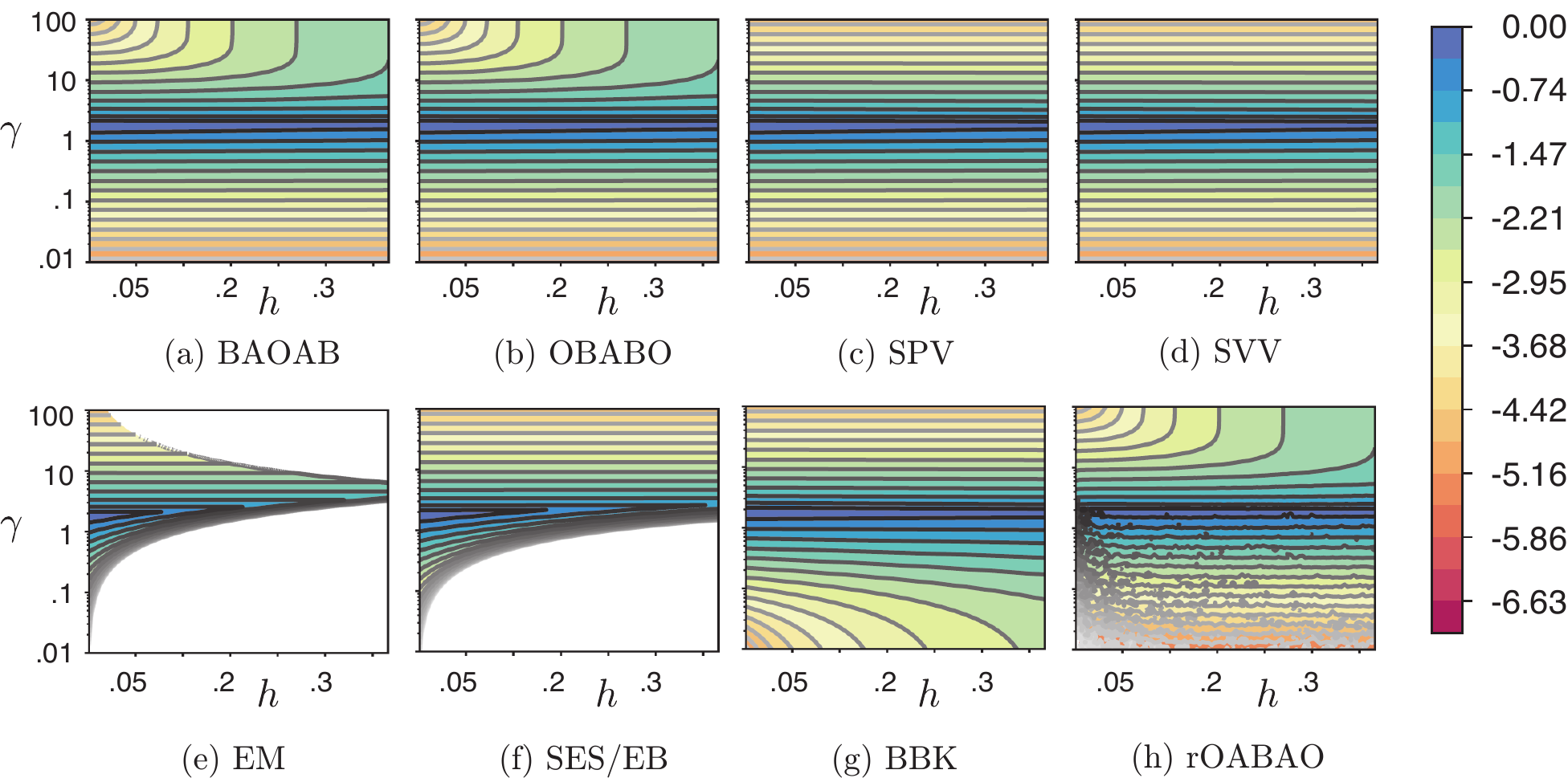}
        \caption{Contour plots of $\ln{\left(\frac{1 - c(h)}{h}\right)}$ for various schemes in the case of an anisotropic Gaussian with parameters $m = 1$ and $M = 10$. Regions of white indicate instability. The rOABAO contour plot is approximate and all other plots are exact (analytic).}
        \label{fig:fig}
\end{figure}

Figure \ref{fig:fig} illustrates the exact synchronously coupled contraction rates for all the numerical integrators we consider (apart from for rOABAO, which is an approximate Monte Carlo estimation) for a range of stepsizes $h$ and friction parameters $\gamma$. BAOAB, OBABO, rOABAO fail to approximate the true kinetic Langevin dynamics for large stepsizes, but still have low bias in the invariant measure as they act like overdamped Langevin dynamics. The $\gamma$-limit convergent property is reflected in Figure \ref{fig:fig} for large $\gamma$ as BAOAB, OBABO and rOABAO have large contraction rates for large values of the stepsize and no longer scale with $1/\gamma$, like the other schemes. SVV, SPV and BBK remain stable, but have convergence rates which scale with $1/\gamma$, indicated by the parallel contour lines in Figure \ref{fig:fig}. The SES and EM methods have large regions of instability, SES being unstable for small values of the friction parameter when $h$ scales larger than $\gamma$.

We have only illustrated convergence results towards the invariant measure, there has been work which provides Wasserstein bias estimates for a few of the numerical methods explored (see \cite{monmarche2021high,monmarche2022hmc,sanz2021wasserstein}). Although the focus of this article is to provide convergence rate estimates, we will provide a comparative numerical study of the bias of each of these numerical methods for some  choices of the friction parameter for an application in the following section.

\subsection{Bayesian Logistic Regression on MNIST}\label{sec:BLL}

We next consider a more involved example, which has a $\nabla$-Lipschitz and convex potential. This is a Bayesian posterior sampling application in multinomial logistic regression using the MNIST machine learning data set \cite{lecun2010mnist}. The data set contains $60,000$ training data points and $10,000$ test data points. The images are of size $28$ by $28$ pixels and hence can be represented in $\mathbb{R}^{784}$. However, we will consider the reduced problem of classifying digits $3$ and $5$.  Sample images are shown in Figure \ref{fig:MNIST}.

\begin{figure}[H]
    \centering
    \begin{minipage}{\textwidth}
        \centering
        \includegraphics[width=0.18\textwidth]{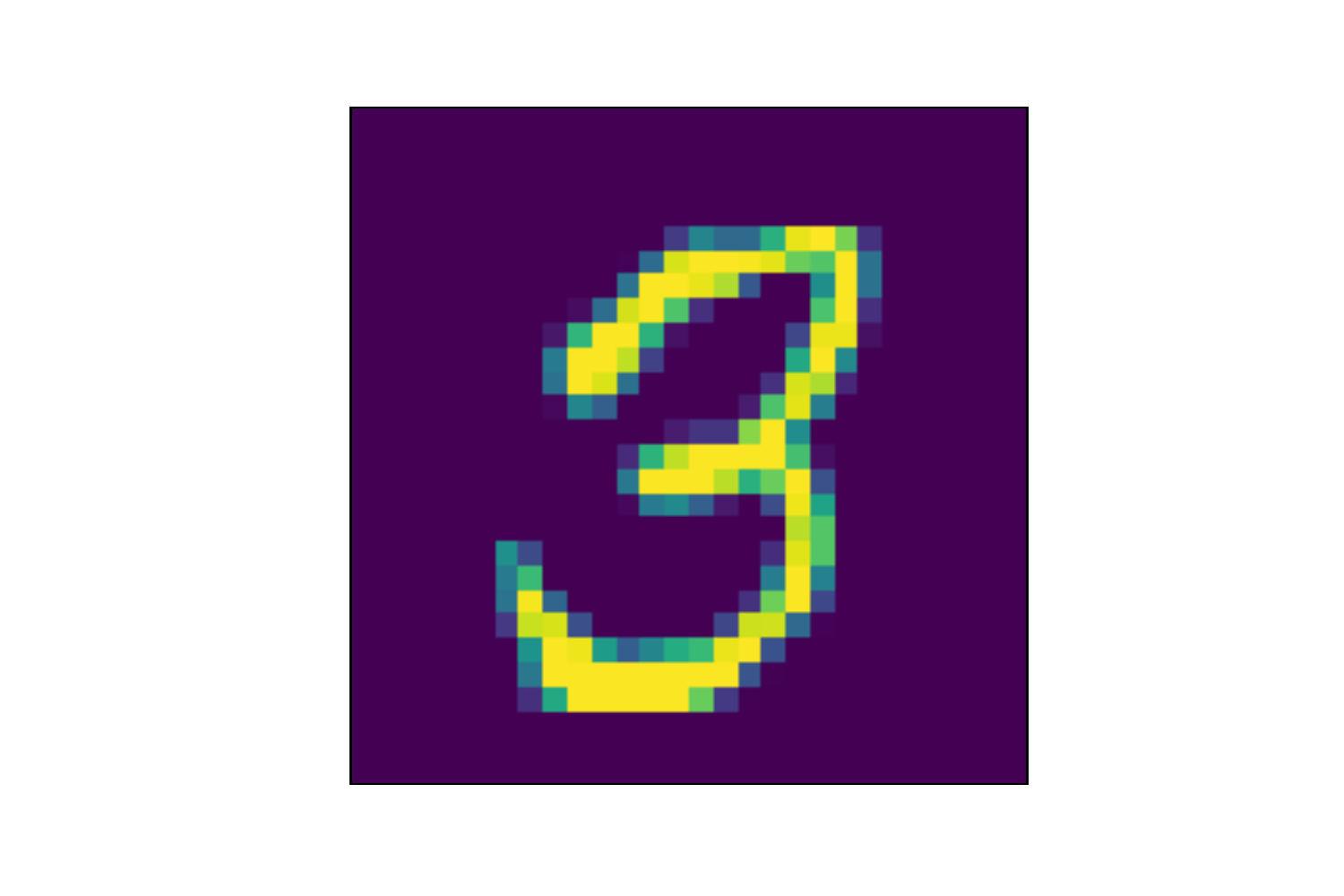}
        \includegraphics[width=0.18\textwidth]{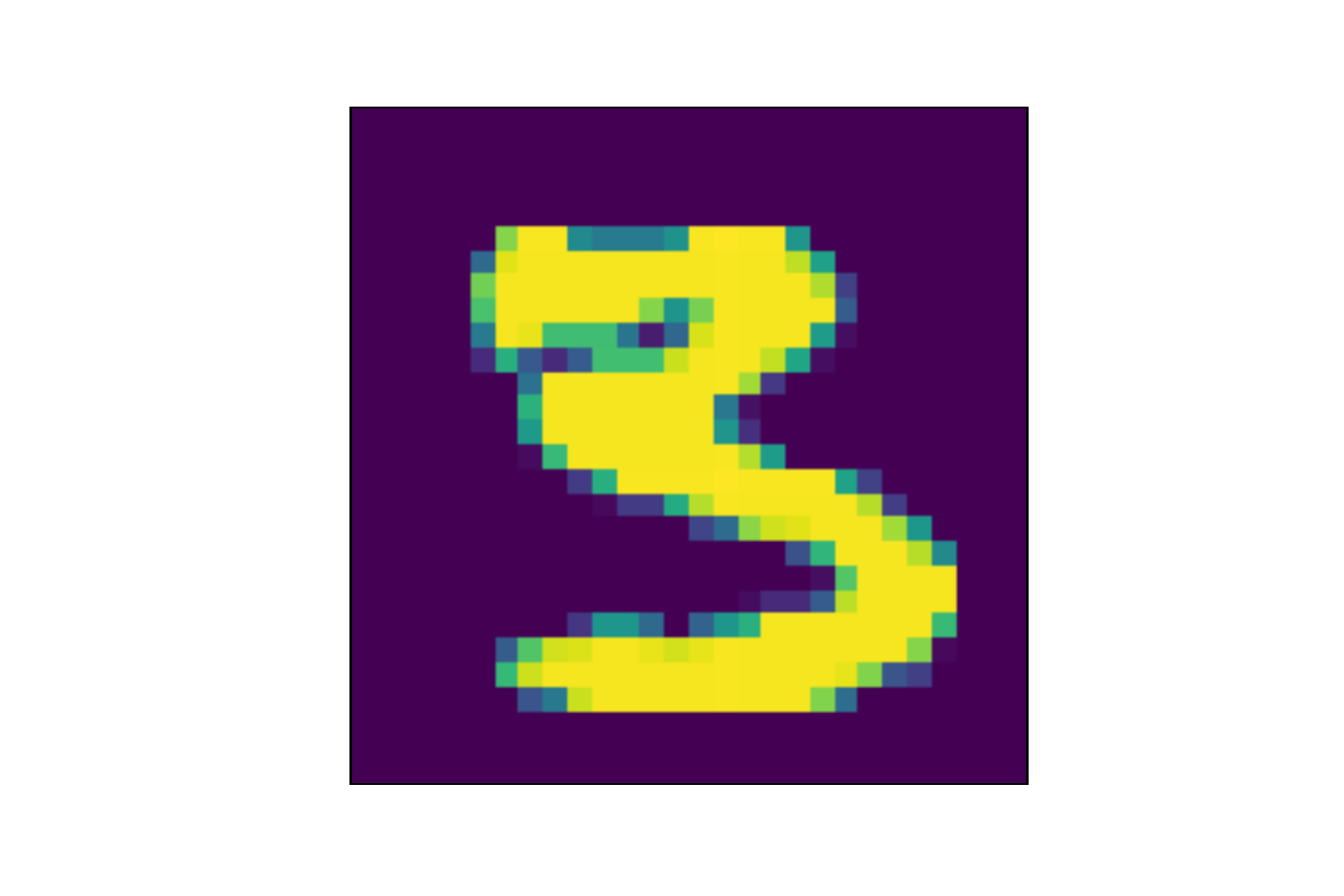}
        \includegraphics[width=0.18\textwidth]{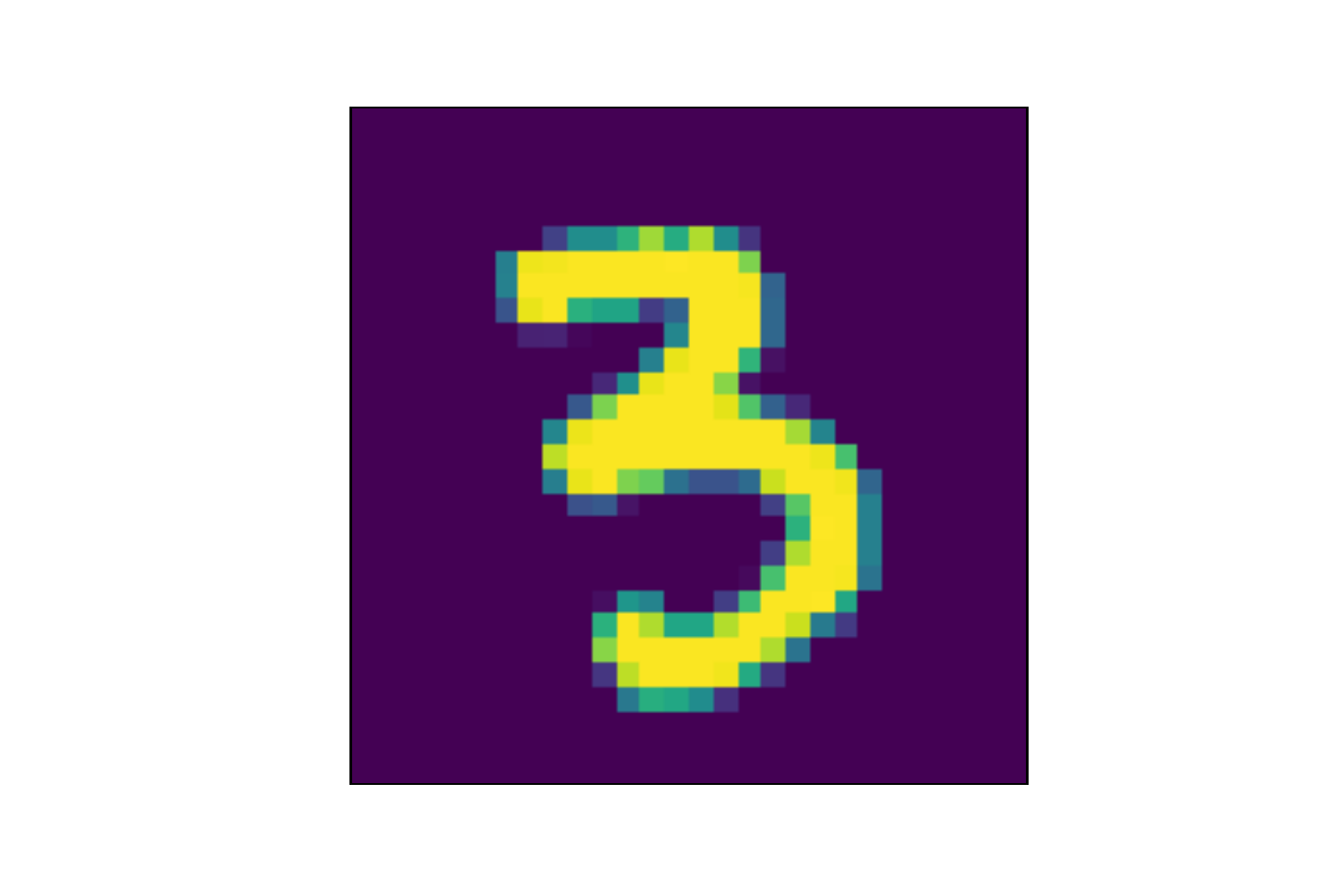}
        \includegraphics[width=0.18\textwidth]{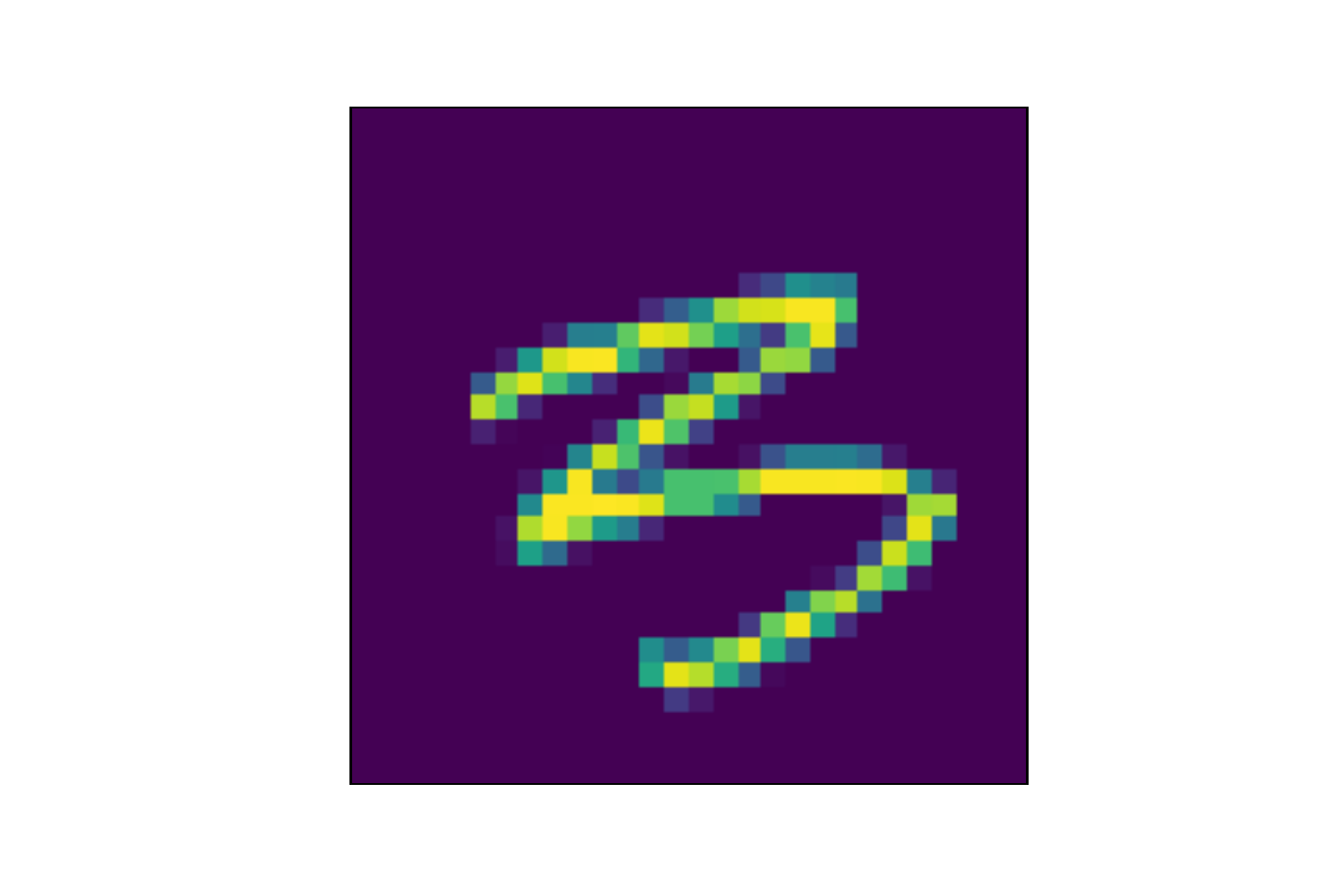}
        \includegraphics[width=0.18\textwidth]{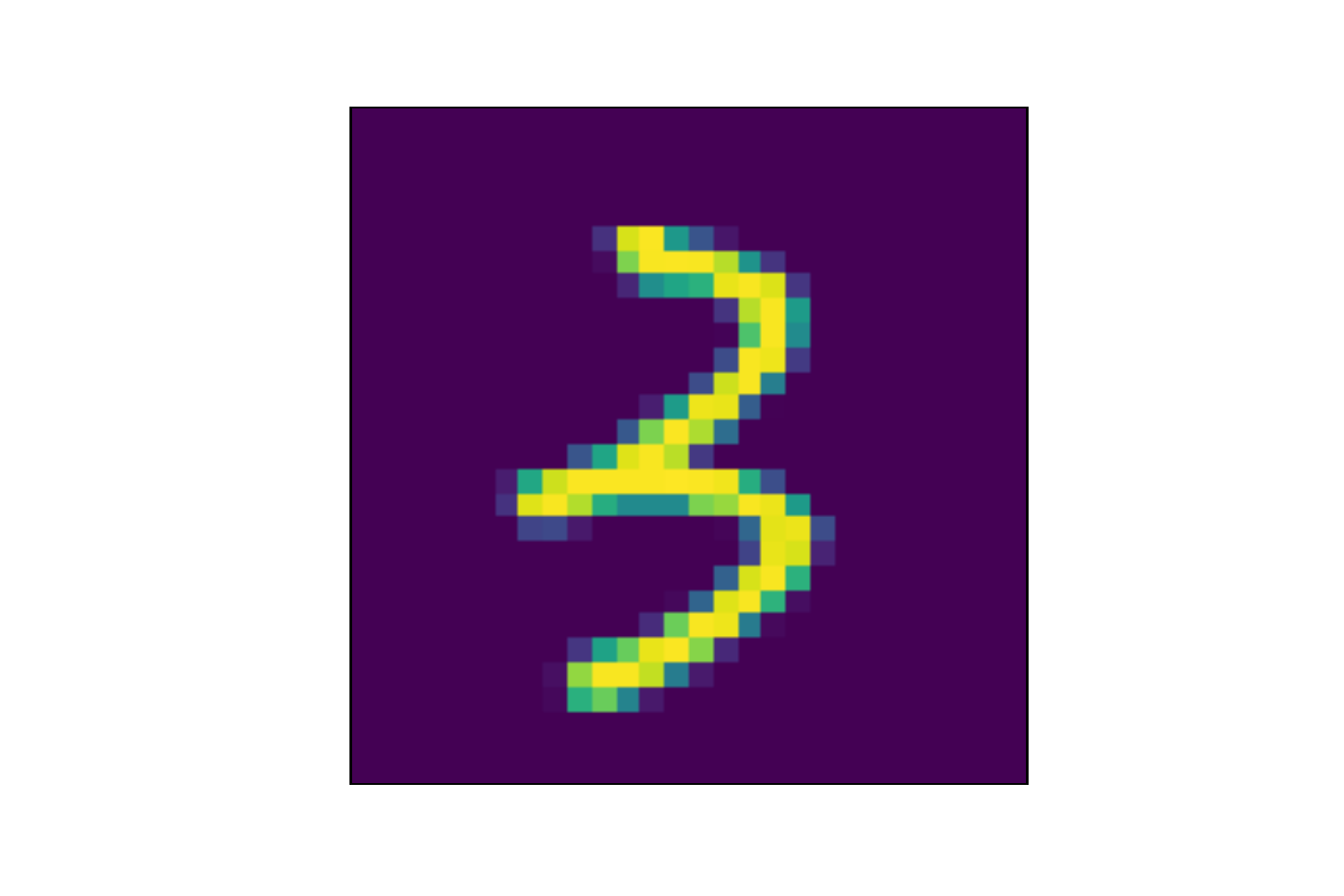}
    \end{minipage}\hfill
    \begin{minipage}{\textwidth}
        \centering
        \includegraphics[width=0.18\textwidth]{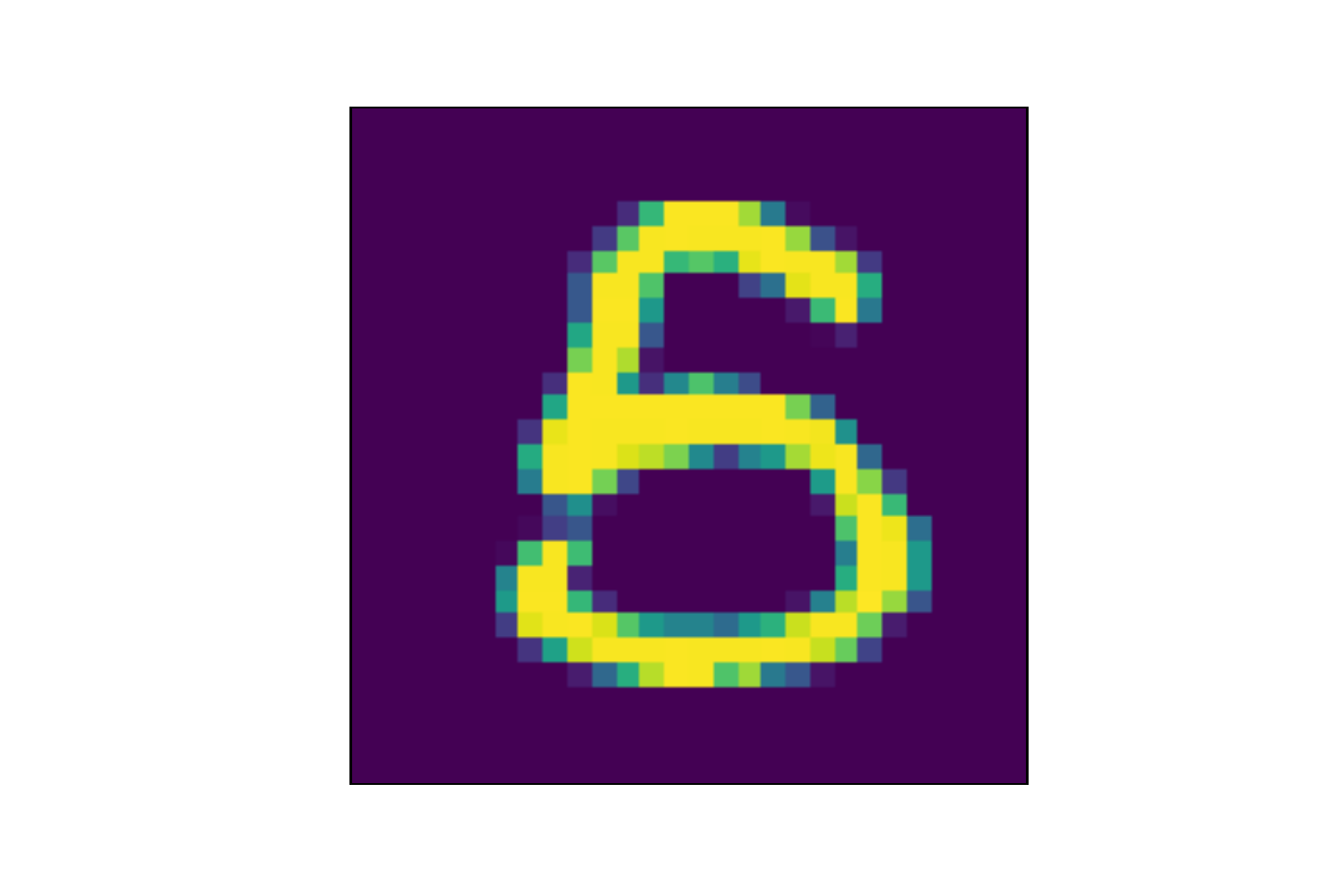}
        \includegraphics[width=0.18\textwidth]{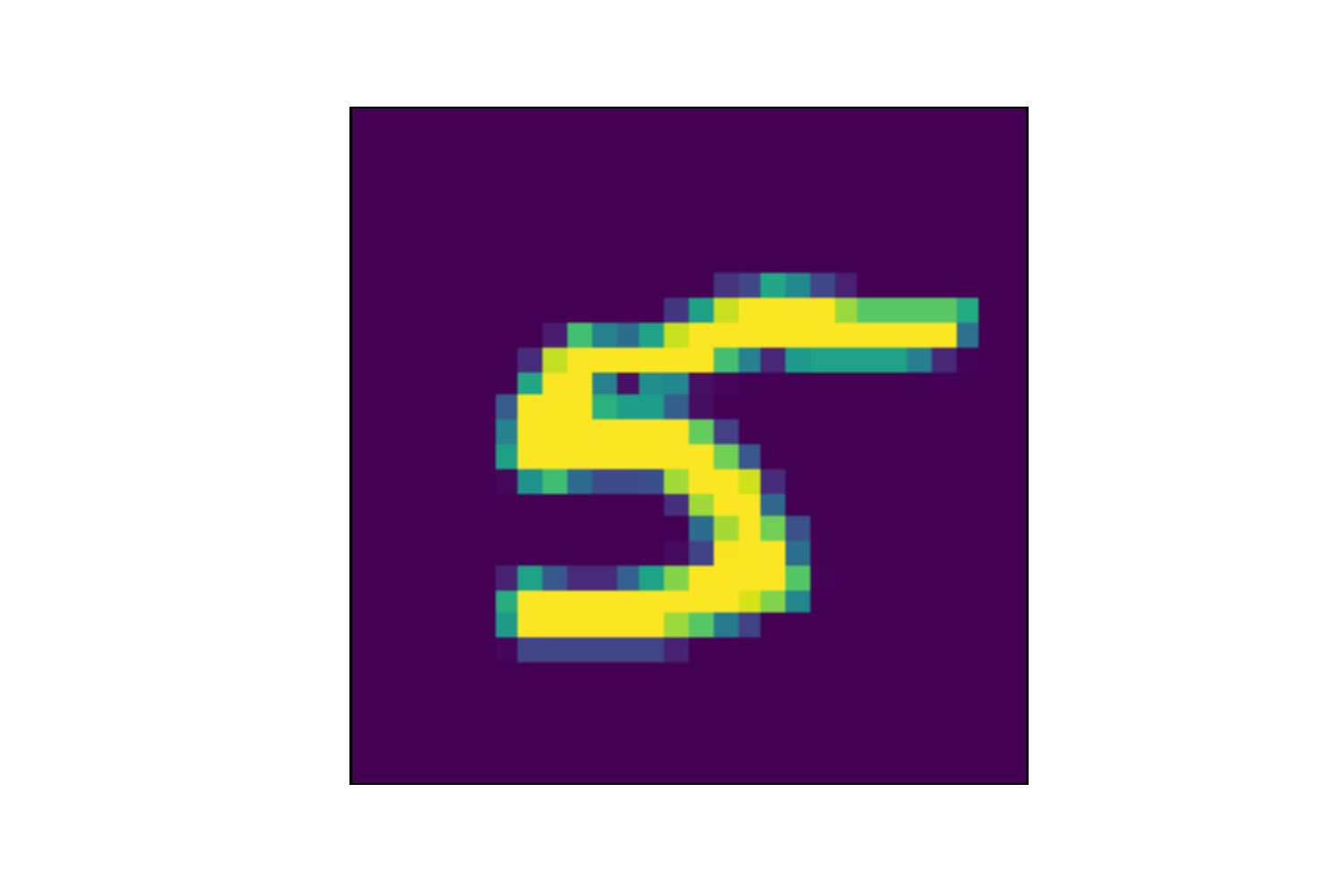}
        \includegraphics[width=0.18\textwidth]{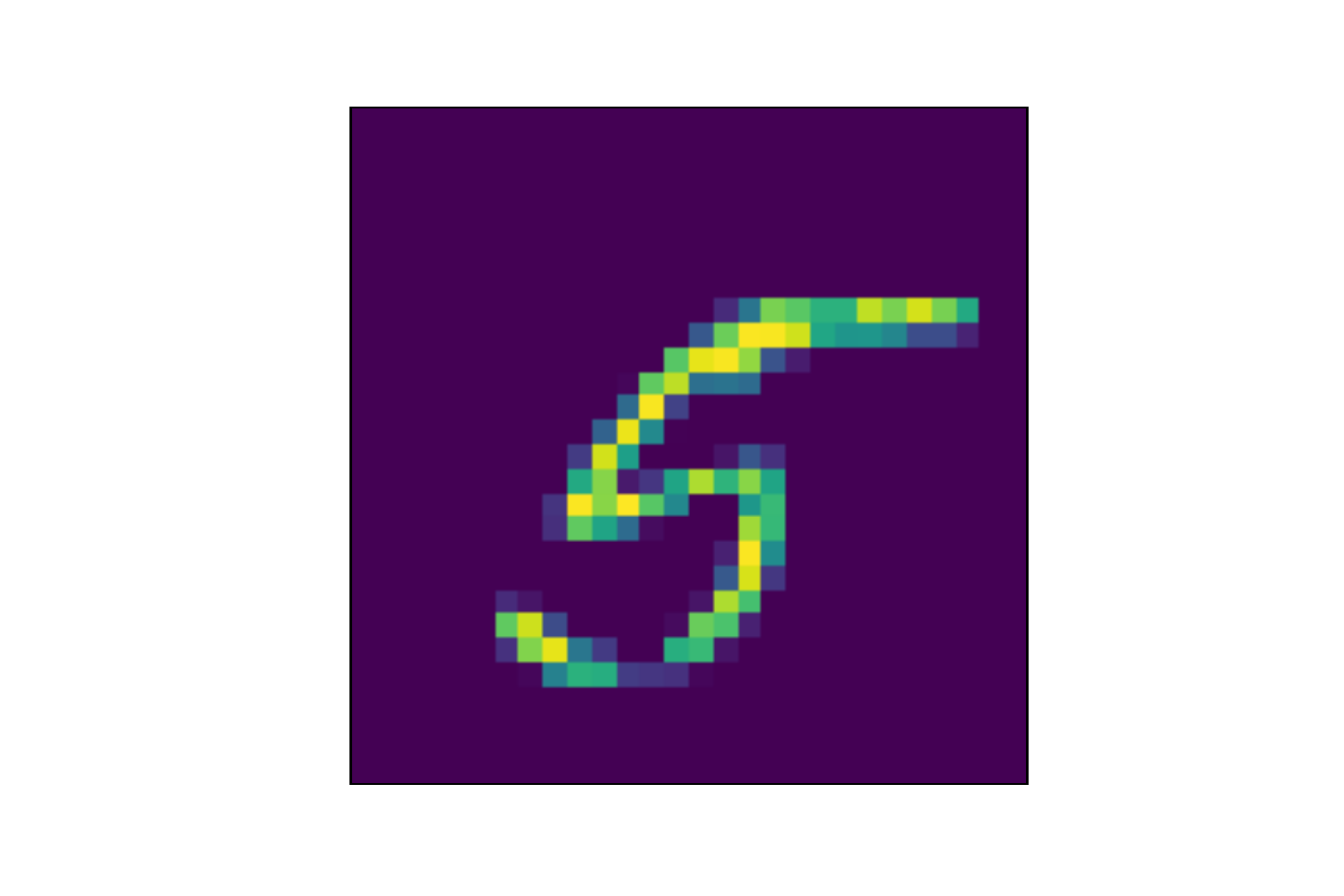}
        \includegraphics[width=0.18\textwidth]{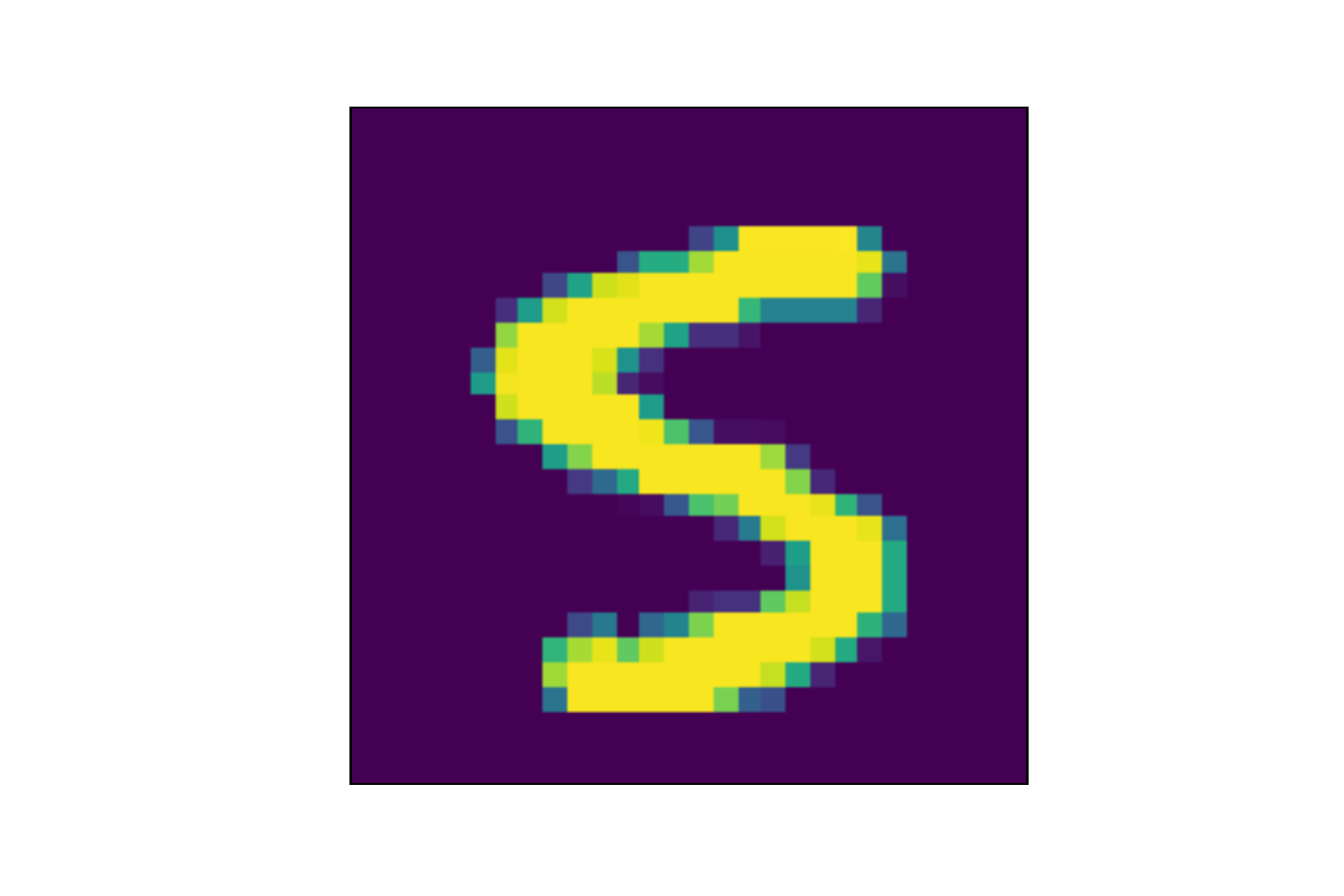}
        \includegraphics[width=0.18\textwidth]{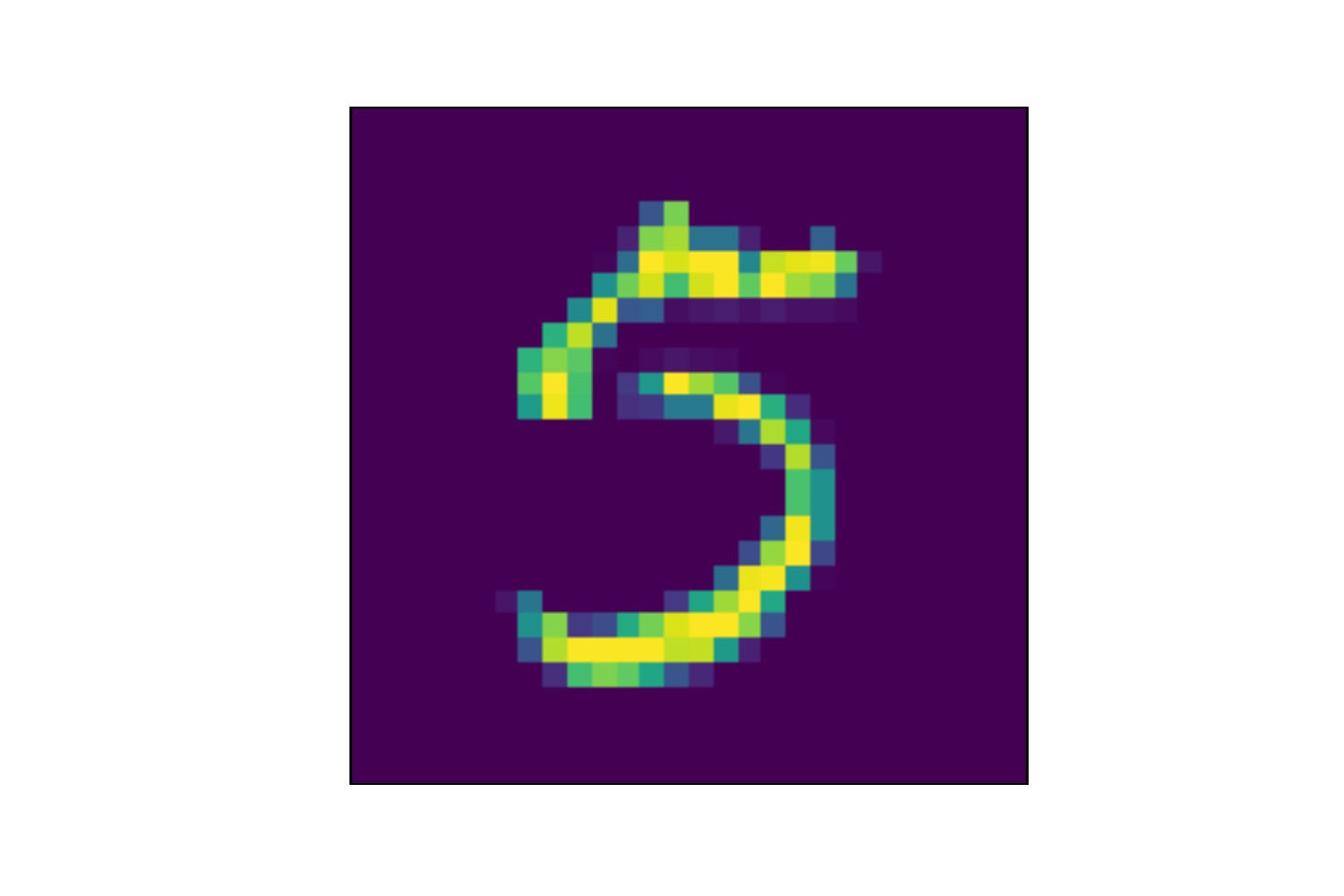}
    \end{minipage}
    \caption{MNIST $3$ and $5$ digits.}
    \label{fig:MNIST}
\end{figure}

We use a i.i.d. Gaussian prior $p_{0}$ with mean 0 and variance $\sigma^{2} = 0.001$. The likelihood function for logistic regression is
\[
p(y^{j}\mid x^{j}, \q) = \frac{\exp{\left(y^{j}\left<x^{j}, \q\right>\right)}}{1+\exp{\left(y^{j}\left<x^{j}, \q\right>\right)}},
\]
where there are $2$ classes (i.e. $y^j$ can take values $0$ and $1$, with 1 corresponding to digit 5, and 0 corresponding to digit 3) and $(x^{j},y^{j})^{N}_{j = 1}$ are the respective training points and labels for a data set of size $N$ (there are $N=11552$ training images of 3 or 5). We then define the posterior potential by
\begin{equation}\label{eq:BLL_potential}
  U(\q) = -\log{\left(p_{0}(\q)\right)} - \sum^{N}_{i = 1} \log{\left( p\left(y^{j}\mid x^{j},\q\right)\right)}.   
\end{equation}
A commonly used method in machine learning and other fields is rely on a stochastic gradient approximation, an unbiased estimator of the gradient of the potential defined in \eqref{eq:BLL_potential}. This is typically obtained based on a sub-sample of size $b$ of a data set of size $N$, where $b << N$, i.e. for a random selection $I_{b} \subset I_{N} := \{1,...,N\}$ one would consider the gradient of
\begin{equation}\label{eq:BLL_SG}
  \widehat{\nabla_{\q} U(\q)}_{SG} = -\nabla_{\q}\log{\left(p_{0}(\q)\right)} - \frac{N}{b}\sum_{i \in I_{b}} \nabla_{\q}\log{\left( p\left(y^{j}\mid x^{j},\q\right)\right)},  
\end{equation}
where the sub-samples are chosen i.i.d at each gradient evaluation (or iteration) of the algorithm. Let $W=I_{b}$ as defined above, and $\mathcal{G}(\q,W)=\widehat{\nabla_{\q} U(\q)}_{SG}$, then it is easy to see that the conditions of Definition \ref{def:stochastic_gradient} hold.

One more accurate estimator for the gradient is the variance reduced stochastic gradient (\cite{johnson2013accelerating}), also called the control variate method in the context of MCMC (see \cite{quiroz2018speeding,baker2019control}). This estimator uses the minimizer (or an approximation) $\q_{\min}$, and estimates the gradient as
\begin{align}\label{eq:BLL_VRSG}
  \widehat{\nabla_{\q} U(\q)}_{VRSG} &= -\nabla_{\q}\log{\left(p_{0}(\q)\right)} - \nabla_{\q_{\min}}\sum^{N}_{i = 1} \log{\left( p\left(y^{j}\mid x^{j},\q_{\min}\right)\right)}\\
  \nonumber&-\frac{N}{b}\sum_{i \in I_{b}} \left[\nabla_{\q}\log{\left( p\left(y^{j}\mid x^{j},\q\right)\right)}-\nabla_{\q_{\min}}\log{\left( p\left(y^{j}\mid x^{j},\q_{\min}\right)\right)}\right].
\end{align}
This can be also shown to satisfy Definition \ref{def:stochastic_gradient} with $W=I_b$ and $\mathcal{G}(\q,W)= \widehat{\nabla_{\q} U(\q)}_{VRSG}$. Both \eqref{eq:BLL_SG} and \eqref{eq:BLL_VRSG} are unbiased estimators of the gradient. In situations where the distribution is concentrated near the minimizer (as the sample size is large compared to the number of parameters, or the prior is sufficiently strong), the \eqref{eq:BLL_VRSG} approximation has much smaller variance, and we found that this reduces the bias of sampling algorithms.
In the following numerics we first consider full gradients for each scheme. We also implemented variance reduced stochastic gradients for BAOAB, based on \eqref{eq:BLL_VRSG} with batch size $b=100$.

We minimized the potential based on the BFGS algorithm, and computed the smallest and largest eigenvalues of the Hessian at the minimizer, which were $m=10^3$ and $M=1.7342\cdot 10^5$. Note that computing the upper and lower bounds on the Hessian globally is not easy for this problem, so we used these eigenvalues at the minimizer instead for setting the parameters in our simulations.
We tried two different friction parameters: $\gamma=\sqrt{M}$ (the lowest value of $\gamma$ for which our theory works) and $\gamma=\sqrt{m}$ (a good choice based on the contraction rates for Gaussians shown on Figure \ref{fig:fig}). 
In terms of stepsize, we tried $h\in \{2/\sqrt{M}, 1/\sqrt{M}, 1/(2\sqrt{M}), 1/(4\sqrt{M})\}$. The stepsize $h=2/\sqrt{M}$ is near the anticipated stability threshold of these methods, this is confirmed by the fact that a larger stepsize ($h=4/\sqrt{M}$) resulted in unstable behaviour and biases above $10^3$ for all methods. 

We used the potential $U$ as a test function, which is often a good choice for examining convergence of Markov chains. The ground truth posterior mean of $U$ was established based on running a well-tuned HMC with accept/reject steps (400 parallel runs, 440 million gradient evaluations in total, with $10\%$ burn-in), this had a standard deviation of $0.023$. The posterior standard deviation of $U$ was also estimated based on these samples, it was found to be $19.82$. 

All tested methods were run in parallel 80 times for 120000 iterations per run (20000 burn-in, 100000 samples), initiated from the minimum of the potential. We computed effective sample sizes based on the approach of \cite{vats2019multivariate}, using the Matlab package \url{https://github.com/lacerbi/multiESS}. All methods were implemented in Matlab on a desktop computer using GPU acceleration.
\begin{table}[H]
\begin{center}
\begin{tabular}{ |c|c|c|c|c| } 
\hline
Algorithm& $\begin{matrix} h=2/\sqrt{M}\\ \gamma=\sqrt{M} \end{matrix}$ & $\begin{matrix} h=1/\sqrt{M}\\ \gamma=\sqrt{M} \end{matrix}$ & $\begin{matrix} h=1/(2\sqrt{M})\\ \gamma=\sqrt{M} \end{matrix}$  & $\begin{matrix} h=1/(4 \sqrt{M}) \\ \gamma=\sqrt{M} \end{matrix}$ \\
\hline
EM & $4.2 (\pm 0.089)$ & $1.5 (\pm 0.13)$ & $0.79 (\pm 0.18)$ & $0.28 (\pm 0.23)$\\
\hline
BBK & $2.7 (\pm 0.061)$ & $0.67 (\pm 0.099)$ & $0.016 (\pm 0.14)$ & $-0.18 (\pm 0.2)$\\
\hline
SPV & $123 (\pm 0.079)$ & $32.1 (\pm 0.091)$ & $8.19 (\pm 0.13)$ & $2.07 (\pm 0.18)$\\
\hline
SVV & $126 (\pm 0.097)$ & $32.8 (\pm 0.091)$ & $8.17 (\pm 0.13)$ & $2.03 (\pm 0.17)$\\
\hline
BAOAB & $-0.043 (\pm 0.049)$ & $-0.002 (\pm 0.058)$ & $0.13 (\pm 0.086)$ & $-0.055 (\pm 0.12)$\\
\hline
BAOAB VRSG & $0.47 (\pm 0.043)$ & $0.23 (\pm 0.066)$ & $0.035 (\pm 0.087)$ & $0.036 (\pm 0.12)$\\
\hline
OBABO & $2.7 (\pm 0.056)$ & $0.67 (\pm 0.076)$ & $0.22 (\pm 0.13)$ & $0.17 (\pm 0.19)$\\
\hline
rOABAO & $-2.6 (\pm 0.062)$ & $-0.61 (\pm 0.094)$ & $0.025 (\pm 0.13)$ & $-0.16 (\pm 0.19)$\\ 
\hline
SES/EB  & $2.6 (\pm 0.072)$ & $1.2 (\pm 0.094)$ & $0.71 (\pm 0.11)$ & $0.2 (\pm 0.18)$\\
\hline
\end{tabular}
\end{center}
\caption{Bias for potential function, $\gamma=\sqrt{M}$}
\label{Table:bias:M}
\end{table}

\begin{table}[H]
\begin{center}
\begin{tabular}{|c|c|c|c|c|} 
\hline
Algorithm& $\begin{matrix} h=2/\sqrt{M}\\ \gamma=\sqrt{m} \end{matrix}$ & $\begin{matrix} h=1/\sqrt{M}\\ \gamma=\sqrt{m} \end{matrix}$ & $\begin{matrix} h=1/(2\sqrt{M})\\ \gamma=\sqrt{m} \end{matrix}$  & $\begin{matrix} h=1/(4 \sqrt{M}) \\ \gamma=\sqrt{m} \end{matrix}$ \\
\hline
EM & $6.4\cdot 10^4 (\pm 0.82)$ & $1.5\cdot 10^4(\pm 0.72)$ & $1.1\cdot 10^3 (\pm 0.73)$ & $4.9 (\pm 0.11)$\\
\hline
BBK & $2.8 (\pm 0.034)$ & $0.68 (\pm 0.041)$ & $0.1 (\pm 0.05)$ & $0.0038 (\pm 0.066)$\\
\hline
SPV & $0.72 (\pm 0.036)$ & $0.14 (\pm 0.043)$ & $0.06 (\pm 0.054)$ & $-0.014 (\pm 0.073)$\\
\hline
SVV & $3.5 (\pm 0.036)$ & $0.81 (\pm 0.043)$ & $0.26 (\pm 0.061)$ & $0.05 (\pm 0.089)$\\
\hline
BAOAB & $0.03 (\pm 0.038)$ & $-0.011 (\pm 0.049)$ & $-0.046 (\pm 0.062)$ & $0.043 (\pm 0.074)$\\
\hline
BAOAB VRSG & $6.4 (\pm 0.04)$ & $2.4 (\pm 0.051)$ & $1.1 (\pm 0.063)$ & $0.55 (\pm 0.075)$\\
\hline
OBABO & $2.7 (\pm 0.032)$ & $0.65 (\pm 0.041)$ & $0.22 (\pm 0.052)$ & $0.11 (\pm 0.071)$\\
\hline
rOABAO & $-1.7 (\pm 0.041)$ & $-0.55 (\pm 0.041)$ & $-0.2 (\pm 0.054)$ & $-0.033 (\pm 0.081)$\\ 
\hline
SES/EB  & $6.0\cdot 10^4 (\pm 0.61)$ & $1.5\cdot 10^4 (\pm 0.48)$ & $1.1\cdot 10^3 (\pm 0.59)$ & $4.7 (\pm 0.068)$\\
\hline
\end{tabular}
\end{center}
\caption{Bias for potential function, $\gamma=\sqrt{m}$}
\label{Table:bias:m}
\end{table}

\begin{table}[H]
\begin{center}
\begin{tabular}{ |c|c|c|c|c| } 
\hline
Algorithm& $\begin{matrix} h=2/\sqrt{M}\\ \gamma=\sqrt{M} \end{matrix}$ & $\begin{matrix} h=1/\sqrt{M}\\ \gamma=\sqrt{M} \end{matrix}$ & $\begin{matrix} h=1/(2\sqrt{M})\\ \gamma=\sqrt{M} \end{matrix}$  & $\begin{matrix} h=1/(4 \sqrt{M}) \\ \gamma=\sqrt{M} \end{matrix}$ \\
\hline
EM   & $146 (\pm 0.7)$ & $221 (\pm 0.998)$ & $282 (\pm 0.822)$ & $327 (\pm 0.581)$\\
\hline
BBK & $85 (\pm 0.535)$ & $148 (\pm 0.726)$ & $221 (\pm 0.969)$ & $285 (\pm 0.933)$\\
\hline
SPV & $86.7 (\pm 0.554)$ & $148 (\pm 0.775)$ & $221 (\pm 0.887)$ & $284 (\pm 0.992)$\\
\hline
SVV & $86.5 (\pm 0.645)$ & $147 (\pm 0.801)$ & $222 (\pm 0.916)$ & $283 (\pm 0.825)$\\
\hline
BAOAB & $44.3 (\pm 0.304)$ & $88.7 (\pm 0.585)$ & $152 (\pm 0.812)$ & $228 (\pm 0.822)$\\
\hline
BAOAB VRSG & $44.6 (\pm 0.332)$ & $86.8 (\pm 0.578)$ & $152 (\pm 0.915)$ & $226 (\pm 0.934)$\\
\hline
OBABO & $68.6 (\pm 0.491)$ & $140 (\pm 0.84)$ & $218 (\pm 0.942)$ & $282 (\pm 0.809)$\\
\hline
rOABAO & $68.5 (\pm 0.507)$ & $140 (\pm 0.692)$ & $219 (\pm 0.781)$ & $283 (\pm 0.862)$\\ 
\hline
SES/EB  & $87.4 (\pm 0.593)$ & $149 (\pm 0.663)$ & $220 (\pm 0.831)$ & $284 (\pm 0.809)$\\
\hline
\end{tabular}
\end{center}
\caption{Gradient evaluations / ESS (potential function), $\gamma=\sqrt{M}$}
\label{Table:gradperess:M}
\end{table}

\begin{table}[H]
\begin{center}
\begin{tabular}{ |c|c|c|c|c| } 
\hline
Algorithm& $\begin{matrix} h=2/\sqrt{M}\\ \gamma=\sqrt{m} \end{matrix}$ & $\begin{matrix} h=1/\sqrt{M}\\ \gamma=\sqrt{m} \end{matrix}$ & $\begin{matrix} h=1/(2\sqrt{M})\\ \gamma=\sqrt{m} \end{matrix}$  & $\begin{matrix} h=1/(4 \sqrt{M}) \\ \gamma=\sqrt{m} \end{matrix}$ \\
\hline
EM & N.A. & N.A. & N.A. & $189 (\pm 0.955)$\\
\hline
BBK  & $15 (\pm 0.124)$ & $30.1 (\pm 0.233)$ & $57.5 (\pm 0.352)$ & $108 (\pm 0.717)$\\
\hline
SPV & $15.1 (\pm 0.106)$ & $29.7 (\pm 0.209)$ & $57.4 (\pm 0.408)$ & $109 (\pm 0.725)$\\
\hline
SVV & $15 (\pm 0.121)$ & $29.9 (\pm 0.222)$ & $57.5 (\pm 0.341)$ & $108 (\pm 0.628)$\\
\hline
BAOAB & $18.8 (\pm 0.128)$ & $36.4 (\pm 0.288)$ & $66.4 (\pm 0.461)$ & $116 (\pm 0.849)$\\
\hline
BAOAB VRSG  & $19.7 (\pm 0.169)$ & $36.4 (\pm 0.242)$ & $67.8 (\pm 0.447)$ & $114 (\pm 0.662)$\\
\hline
OBABO & $15 (\pm 0.118)$ & $30 (\pm 0.204)$ & $57.5 (\pm 0.471)$ & $108 (\pm 0.711)$\\
\hline
rOABAO & $16.5 (\pm 0.236)$ & $29.7 (\pm 0.218)$ & $58.2 (\pm 0.356)$ & $109 (\pm 0.669)$\\ 
\hline
SES/EB & N.A. & N.A. & N.A. & $108 (\pm 0.652)$  \\
\hline
\end{tabular}
\end{center}
\caption{Gradient evaluations / ESS (potential function), $\gamma=\sqrt{m}$. N.A. indicates that the method did not converge for the given stepsize.}
\label{Table:gradperess:m}
\end{table}
Firstly, when changing from $\gamma=\sqrt{M}$ to $\gamma=\sqrt{m}$, we can see that the changes in bias are not significant for BAOAB, OBABO, rOABAO, and BBK, the bias increases significantly for EM and SES (instability issues) and somewhat for BAOAB VRSG, and the bias decreases significantly for SPV and SVV. In terms of gradient evaluations per ESS, the choice $\gamma=\sqrt{m}$ is more efficient by a factor of 2-6 for all methods except EM and SES. This is in line with the recent research in accelerated convergence rates for underdamped Langevin dynamics (\cite{zhang2023improved,cao2019explicit}).

We can see that BAOAB has impressively low bias for the potential test function even at the largest stepsize $2/\sqrt{M}$, and it also has a competitive computational cost in terms of gradient evaluations / effective sample size (ESS). The VRSG variant of BAOAB has somewhat larger bias (especially at lower frictions), but it requires a similar number of iterations per ESS, with much lower computational cost per iteration compared to using full gradients. The rOABAO scheme based on randomized midpoints has a relatively low bias at all stepsizes, and requires a rather small number of gradient evaluations per iteration. It is beyond the scope of this paper, but we think that more significant differences could arise between these schemes for less smooth potentials.

\section{Conclusion}
In this article we have extended the results of \cite{leimkuhler2023contraction} to  further integration schemes. By building stepsize-dependent norms we achieve convergence rates which hold on a large interval of stepsize, in many cases the same as the stability threshold of the numerical method (up to a constant factor). We further considered the case of stochastic gradients, where we allow a flexible choice of unbiased gradient estimator under the assumption that the expected variance of the Jacobian of the estimator is bounded. We show that this results in a reduced convergence rate based on the variance of the Jacobian of the estimator, which coincides with what we have observed numerically for small batch sizes in a subsampled stochastic gradient. We have provided numerical results comparing the bias of each of the numerical methods based on choices of the friction parameter which are optimal according to our theory or the optimal choice for the Gaussian distribution, where we solved for the convergence rates exactly. We compared the errors of the integrators in a Bayesian logistic regression application and have seen that some of the integrators performed well with large stepsizes, even in the presence of stochastic gradients.

\section*{Acknowledgments}
The authors acknowledge the support of the Engineering and Physical Sciences Research Council Grant EP/S023291/1 (MAC-MIGS
Centre for Doctoral Training). The authors thank Sinho Chewi for the valuable discussion about acceleration in kinetic Langevin dynamics.

\bibliographystyle{amsplain}
\bibliography{references}

\appendix
\section{Stochastic Gradient Kinetic Langevin Dynamics Integrators}\label{Appendix:SG-Algorithms}
For the Euler-Maruyama, stochastic Euler scheme, rOABAO, stochastic position Verlet only one force evaluation is used in each iteration, so every gradient evaluation is taken as a stochastic gradient estimate. The complete algorithms are stated below in Algorithms \ref{alg:SG-EM},\ref{alg:SG-SES}, \ref{alg:SG-sOABAO} and \ref{alg:SG-SPV}.

\begin{algorithm}[H]
    \footnotesize
    \SetAlgoLined
	\begin{itemize}
		\item Initialize $\left(x_{0},v_{0}\right) \in \mathbb{R}^{2n}$, stepsize $h > 0$ and friction parameter $\gamma > 0$.
            \item Sample $G \sim \mathcal{G}(x,\cdot)$ 
		\item for $k = 1,2,...,K$ do
		\begin{itemize}
  
                \item[] Sample $\xi \sim \mathcal{N}(0,1)^{n}$
                \item[] $x_{k} \to x_{k-1} + h v_{k-1}$
                \item[] $v_{k} \to v_{k-1} - h G - h \gamma v_{k-1} + \sqrt{2\gamma h}\xi$
                \item[] Sample $G \sim \mathcal{G}(x_{k},\cdot)$ 
                
		\end{itemize}
            \item Output: Samples $(x_{k})^{K}_{k=0}$.
	\end{itemize}
            
	\caption{Stochastic Gradient Euler-Maruyama (EM)}
	\label{alg:SG-EM}
\end{algorithm}

\begin{algorithm}[H]
    \footnotesize
    \SetAlgoLined
	\begin{itemize}
		\item Initialize $\left(x_{0},v_{0}\right) \in \mathbb{R}^{2n}$, stepsize $h > 0$ and friction parameter $\gamma > 0$.
            \item Sample $G \sim \mathcal{G}(x,\cdot)$ 
		\item for $k = 1,2,...,K$ do
		\begin{itemize}
  
                \item[] Sample $(\zeta,\omega) \sim \mathcal{N}(0,\Sigma)$, where $\Sigma$ is given in \eqref{eq:SES_covariance}.
                \item[] Sample $G \sim \mathcal{G}(x_{k-1},\cdot)$
                \item[] $x_{k} \to x_{k-1} + \frac{1-\eta}{\gamma}v_{k-1} - \frac{\gamma h + \eta - 1}{\gamma^{2}}G + \zeta$
                \item[] $v_{k} \to \eta v_{k-1} - \frac{1-\eta}{\gamma}G + \omega$
		\end{itemize}
            \item Output: Samples $(x_{k})^{K}_{k=0}$.
	\end{itemize}
            
	\caption{Stochastic Gradient Stochastic Euler Scheme (SES/EB)}
	\label{alg:SG-SES}
\end{algorithm}
\begin{algorithm}[H]
    \footnotesize
    \SetAlgoLined
	\begin{itemize}
		\item Initialize $\left(x_{0},v_{0}\right) \in \mathbb{R}^{2n}$, stepsize $h > 0$ and friction parameter $\gamma > 0$.
            \item Sample $G \sim \mathcal{G}(x_{0},\cdot)$ 
		\item for $k = 1,2,...,K$ do
		\begin{itemize}
                \item[] Sample $u \sim \mathcal{U}(0,h)$
                \item[] Sample $G \sim \mathcal{G}(x_{k-1} + uv_{k-1},\cdot)$
                \item[] Sample $\xi \sim \mathcal{N}(0,1)^{n}$
                \item[(O)] $v \to \eta^{1/2} v_{k-1} + \sqrt{1-\eta}\xi$
                \item[] $x_{k} \to x_{k-1} + hv_{k-1} - \frac{h^{2}}{2}G$
                \item[] $v_{k} \to v_{k-1} - hG$
                \item[] Sample $\xi \sim \mathcal{N}(0,1)^{n}$
                \item[(O)] $v_{k} \to \eta^{1/2} v + \sqrt{1-\eta}\xi$
		\end{itemize}
            \item Output: Samples $(x_{k})^{K}_{k=0}$.
	\end{itemize}
            
	\caption{Stochastic Gradient rOABAO}
	\label{alg:SG-sOABAO}
\end{algorithm}

\begin{algorithm}[H]
    \footnotesize
    \SetAlgoLined
	\begin{itemize}
		\item Initialize $\left(x_{0},v_{0}\right) \in \mathbb{R}^{2n}$, stepsize $h > 0$ and friction parameter $\gamma > 0$.
            \item Sample $G \sim \mathcal{G}(x_{0},\cdot)$ 
		\item for $k = 1,2,...,K$ do
		\begin{itemize}

                \item[(A)] $x \to x_{k-1} + \frac{h}{2}v_{k-1}$
                \item[] Sample $G \sim \mathcal{G}(x,\cdot)$ and $\xi \sim \mathcal{N}(0,1)^{n}$
                \item[($\mathcal{V}_{s}$)] $v_{k} \to \eta v_{k-1}-\frac{1-\eta}{\gamma}G + \sqrt{1 - \eta^{2}}\xi$
                \item[(A)] $x_{k} \to x + \frac{h}{2}v_{k}$
        
		\end{itemize}
            \item Output: Samples $(x_{k})^{K}_{k=0}$.
	\end{itemize}
            
	\caption{Stochastic Gradient Stochastic Velocity Verlet (SVV)}
	\label{alg:SG-SPV}
\end{algorithm}

For BAOAB the first and last $\mathcal{B}$ of each iteration share a stochastic gradient evaluation to make the algorithm roughly one gradient evaluation per step. The complete algorithm is given in Algorithm \ref{alg:SG-BAOAB}.

\begin{algorithm}[H]
    \footnotesize
    \SetAlgoLined
	\begin{itemize}
		\item Initialize $\left(x_{0},v_{0}\right) \in \mathbb{R}^{2n}$, stepsize $h > 0$ and friction parameter $\gamma > 0$.
            \item Sample $G \sim \mathcal{G}(x_{0},\cdot)$ 
		\item for $k = 1,2,...,K$ do
		\begin{itemize}
			\item[(B)] $v \to v_{k-1} - \frac{h}{2}G$
                \item[(A)] $x \to x_{k-1} + \frac{h}{2}v$
                \item[] Sample $\xi \sim \mathcal{N}(0,1)^{n}$
                \item[(O)] $v \to \eta v + \sqrt{1-\eta^{2}}\xi$
                \item[(A)] $x_{k} \to x + \frac{h}{2}v$
                \item[] Sample $G \sim \mathcal{G}(x_{k},\cdot)$ 
                \item[(B)] $v_{k} \to v - \frac{h}{2}G$
		\end{itemize}
            \item Output: Samples $(x_{k})^{K}_{k=0}$.
	\end{itemize}
            
	\caption{Stochastic Gradient BAOAB}
	\label{alg:SG-BAOAB}
\end{algorithm}

Similarly for OBABO the first and last $\mathcal{B}$ of each iteration share a stochastic gradient evaluation to make the algorithm roughly one gradient evaluation per step. The complete algorithm is given in Algorithm \ref{alg:SG-OBABO}.

\begin{algorithm}[H]
    \footnotesize
    \SetAlgoLined
	\begin{itemize}
		\item Initialize $\left(x_{0},v_{0}\right) \in \mathbb{R}^{2n}$, stepsize $h > 0$ and friction parameter $\gamma > 0$.
            \item Sample $G \sim \mathcal{G}(x_{0},\cdot)$ 
		\item for $k = 1,2,...,K$ do
		\begin{itemize}
                \item[] Sample $\xi \sim \mathcal{N}(0,1)^{n}$
                \item[(O)] $v \to \eta^{1/2} v_{k-1} + \sqrt{1-\eta}\xi$
                \item[(B)] $x \to x_{k-1} - \frac{h}{2}G$
                \item[(A)] $x \to x + hv$
                \item[] Sample $G \sim \mathcal{G}(x,\cdot)$ 
                \item[(B)] $x_{k} \to x - \frac{h}{2}G$
                \item[] Sample $\xi \sim \mathcal{N}(0,1)^{n}$
                \item[(O)] $v_{k} \to \eta^{1/2} v + \sqrt{1-\eta}\xi$
		\end{itemize}
            \item Output: Samples $(x_{k})^{K}_{k=0}$.
	\end{itemize}
            
	\caption{Stochastic Gradient OBABO}
	\label{alg:SG-OBABO}
\end{algorithm}
If we express each iteration of the BBK methods as $\Phi_{B_{2}}\circ\Phi_{A}\circ \Phi_{B_{1}}$ as in Section \ref{Sec:Proofs}, then the last $\Phi_{B_{2}}$ step of each iteration and the first $\Phi_{B_{1}}$ of the next iteration share the same stochastic gradient evaluation to make the algorithm roughly one gradient evaluation per step. The complete algorithm is given in Algorithm \ref{alg:SG-BBK}.

\begin{algorithm}[H]
    \footnotesize
    \SetAlgoLined
	\begin{itemize}
		\item Initialize $\left(x_{0},v_{0}\right) \in \mathbb{R}^{2n}$, stepsize $h > 0$ and friction parameter $\gamma > 0$.
            \item Sample $G \sim \mathcal{G}(x_{0},\cdot)$ 
		\item for $k = 1,2,...,K$ do
		\begin{itemize}
                \item[] Sample $\xi \sim \mathcal{N}(0,1)^{n}$
                \item[($B_{1}$)] $v \to v_{k-1} + \frac{h}{2}\left(-G - \gamma v_{k-1} + \frac{\sqrt{2\gamma}}{\sqrt{h}}\xi\right)$
                \item[(A)] $x_{k} \to x_{k-1} + hv$
                \item[] Sample $G \sim \mathcal{G}(x_{k},\cdot)$ and $\xi \sim \mathcal{N}(0,1)^{n}$
                \item[($B_{2}$)] $v_{k} \to \left(1 + \frac{h}{2}\gamma\right)^{-1}\left(v - \frac{h}{2}G + \frac{\sqrt{2\gamma}}{\sqrt{h}}\xi\right)$
		\end{itemize}
            \item Output: Samples $(x_{k})^{K}_{k=0}$.
	\end{itemize}
            
	\caption{Stochastic Gradient Brunger-Brooks Karplus (BBK)}
	\label{alg:SG-BBK}
\end{algorithm}

Finally for SVV the last $\mathcal{V}_{s}$ step of each iteration and the first $\mathcal{V}_{s}$ of the next iteration share the same stochastic gradient evaluation. The complete algorithm is given in Algorithm \ref{alg:SG-SVV}.

\begin{algorithm}[H]
    \footnotesize
    \SetAlgoLined
	\begin{itemize}
		\item Initialize $\left(x_{0},v_{0}\right) \in \mathbb{R}^{2n}$, stepsize $h > 0$ and friction parameter $\gamma > 0$.
            \item Sample $G \sim \mathcal{G}(x_{0},\cdot)$ 
		\item for $k = 1,2,...,K$ do
		\begin{itemize}
                \item[] Sample $\xi \sim \mathcal{N}(0,1)^{n}$
                \item[($\mathcal{V}_{s}$)] $v \to \eta^{1/2} v_{k-1}-\frac{1-\eta^{1/2}}{\gamma}G + \sqrt{1 - \eta}\xi$
                \item[(A)] $x_{k} \to x_{k-1} + hv$
                \item[] Sample $G \sim \mathcal{G}(x_{k},\cdot)$ and $\xi \sim \mathcal{N}(0,1)^{n}$
                \item[($\mathcal{V}_{s}$)] $v_{k} \to \eta^{1/2} v -\frac{1-\eta^{1/2}}{\gamma}G + \sqrt{1 - \eta}\xi$
		\end{itemize}
            \item Output: Samples $(x_{k})^{K}_{k=0}$.
	\end{itemize}
            
	\caption{Stochastic Gradient Stochastic Velocity Verlet (SVV)}
	\label{alg:SG-SVV}
\end{algorithm}

\end{document}